\RequirePackage{ifpdf}          
\ifpdf 
    \documentclass[10pt,pdftex]{article}
    \usepackage{hyperref} 
    \hypersetup{colorlinks, citecolor=MyDarkRed, filecolor=MyDarkBlue,
      linkcolor=MyDarkBlue, urlcolor=MyDarkBlue}
\else 
    \documentclass[10pt,dvips]{article}
    \usepackage{hyperref} 
    \hypersetup{colorlinks, citecolor=MyDarkRed, filecolor=MyDarkBlue,
      linkcolor=MyDarkRed, urlcolor=MyDarkBlue}
    \usepackage{breakurl}     
\fi

\usepackage[in]{fullpage}       
\usepackage{amsmath,amssymb}	
\usepackage{pslatex}		 
\usepackage[usenames]{xcolor}   
\definecolor{MyDarkBlue}{rgb}{0, 0.0, 0.45} 
\definecolor{MyDarkRed}{rgb}{0.45, 0.0, 0} 
\definecolor{MyDarkGreen}{rgb}{0, 0.45, 0} 
\definecolor{MyLightGray}{gray}{.90}
\definecolor{MyLightGreen}{rgb}{0.5, 0.99, 0.5}
\DefineNamedColor{named}{Peach}         {cmyk}{0,0.50,0.70,0}
\DefineNamedColor{named}{Cyan}          {cmyk}{1,0,0,0}

\usepackage{graphicx} 
\usepackage[parfill]{parskip}    
\usepackage{amsfonts, amscd, amssymb, amsthm, amsmath}
\usepackage[mathscr]{euscript}
\theoremstyle{plain}
\newtheorem{thm}{Theorem}[section]
\newtheorem*{theorem-non}{Theorem}
\newtheorem{lem}[thm]{Lemma}
\newtheorem{prop}[thm]{Proposition}

\theoremstyle{definition}

\theoremstyle{remark}
\newtheorem*{rem}{Remark}

\usepackage{color}
\definecolor{dred}{rgb}{.65, 0, 0.15}

  \def\CC{\mathbb{C}}                       \def\ZZ{\mathbb{Z}}

\def\<{\langle} \def\>{\rangle}

\def\cond{\mathrm{cond}\,}
\def\Res{\mathrm{Res}}

\usepackage{listings}           
\pagecolor{white}		

\begin{document}

\title{On the third moment of $L\big(\tfrac{1}{2}, \chi_{\scriptscriptstyle d}\big)$ {II}: \\ 
the number field case}	
\author{Adrian Diaconu\footnote{{School of Mathematics, University of Minnesota, Minneapolis, MN 55455,
Email: cad@umn.edu}} \, 
and Ian Whitehead\footnote{{Department of Mathematics, Statistics, and Computer Science, Macalester College, St. Paul, MN 55105, Email: iwhitehe@umn.edu}}         
}
\date{}	                
\maketitle

\begin{abstract}
\noindent We establish a smoothed asymptotic formula for the third moment of quadratic {D}irichlet $L$-functions at the central value. In addition to the main term, which is known, we prove the existence of a secondary term of size $x^{\frac{3}{4}}$. The error term in the asymptotic formula is on the order of $O(x^{\frac{2}{3}+\delta})$ for every $\delta > 0.$ 
\end{abstract}

\section{Introduction} 

{\bf Statement of the main results.} The object of this sequel of \cite{D} is to study the analytic continuation of the 
{D}irichlet generating series 
\begin{equation} \label{eq: gen-series-to-study}
Z_{\scriptscriptstyle 0}(s)\, =  
\sideset{}{^*}\sum_{d} L\big(\tfrac{1}{2}, \chi_{\scriptscriptstyle 2 d} \big)^{\! \scriptscriptstyle 3} d^{\scriptscriptstyle - s}
\end{equation} 
where the star indicates that the sum is over all square-free odd positive integers, \!associated to the central values 
of quadratic {D}irichlet {$L$}-functions. The series \eqref{eq: gen-series-to-study} is absolutely convergent for complex 
$s$ with sufficiently large real part -- in fact, by a well-known result of Heath-Brown \cite{H-B}, for $\Re(s) > 1.$

Our main result is the following

\vskip5pt
\begin{thm} \label{Main Theorem A} --- The function $Z_{\scriptscriptstyle 0}(s)$ has meromorphic continuation 
to the half-plane $\Re(s) > \frac{2}{3}.$ It is analytic in this region, except for a pole of order seven at $s = 1,$ and 
a simple pole at $s = \frac{3}{4}$ with residue
\begin{equation*}
\underset{s=\frac{3}{4}}{\Res}\; Z_{\scriptscriptstyle 0}(s) = \tfrac{9}{256\pi} 2^{\scriptscriptstyle \frac{1}{4}} (-181+128\sqrt{2}) 
\Gamma\!\left(\tfrac{1}{4}\right)^{\! \scriptscriptstyle 4} \zeta\!\left(\tfrac{1}{2}\right)^{\! \scriptscriptstyle 7} \cdot 
\prod_{p\neq 2\; \mathrm{ prime}} P(p^{\scriptscriptstyle -\frac{1}{2}})\approx-.0034
\end{equation*}
where 
$
P(p^{\scriptscriptstyle -\frac{1}{2}}) = 
\big(1-p^{\scriptscriptstyle -\frac{1}{2}}\big)^{\! 5} \big(1+p^{\scriptscriptstyle -\frac{1}{2}}\big)
\big(1+4p^{\scriptscriptstyle -\frac{1}{2}}+11p^{\scriptscriptstyle -1}+10p^{\scriptscriptstyle -\frac{3}{2}} - 
11p^{\scriptscriptstyle -2}+11p^{\scriptscriptstyle -3}-4p^{\scriptscriptstyle -\frac{7}{2}}-p^{\scriptscriptstyle -4}\big).
$ 

Moreover, for every small $\delta > 0$ and $s\in \mathbb{C}$ with 
$\frac{2}{3} < \Re(s) < 1 + \delta$ and $|\Im(s)| > 1,$ we have the estimate 
\begin{equation*} \label{eq: estimate-Z0}
Z_{\scriptscriptstyle 0}(s) \ll_{\scriptscriptstyle \delta} |s|^{\scriptscriptstyle 5(1 - \Re(s)) \, + \, 6\delta}.
\end{equation*}
\end{thm}

Meromorphic continuation of $Z_{\scriptscriptstyle 0}(s)$ beyond $\Re(s)=1$ and analysis of the principal part at $s=1$ have already been given in \cite{Sound}, \cite{DGH}, and \cite{Young}. Our focus here is on the pole at  $s = \frac{3}{4}$ and further meromorphic continuation to $\Re(s) > \frac{2}{3}.$ As a consequence, we have the following smoothed asymptotic formula for the cubic moment of quadratic {D}irichlet {$L$}-functions.

\vskip5pt
\begin{thm} \label{Main Theorem B} --- Let $W: (0, \infty) \to [0, 1]$ be a smooth function with compact support 
contained in $[1\slash 2, 1],$ and satisfying 
\begin{equation}  \label{eq: weight-function}
|W^{\scriptscriptstyle(j)}(u)| \le 1 \;\;\,  \text{for $0\le j \le 3$ and $u \in \mathbb{R}$}.
\end{equation} 
Letting $\widehat{W}$ denote the Mellin transform of $W,$ then, for every $x \ge 1$ and small $\delta > 0,$ we have 
\begin{equation*} 
\sideset{}{^*}\sum_{d} L\big(\tfrac{1}{2}, \chi_{\scriptscriptstyle 2 d} \big)^{\! \scriptscriptstyle 3}\, 
W\big(\tfrac{d}{x}\big) \, = \, x\, Q_{\scriptscriptstyle W}(\log \, x)\, + \, 
\underset{s = \frac{3}{4}}{\Res} \; Z_{\scriptscriptstyle 0}(s) \cdot
\widehat{W}\big(\tfrac{3}{4}\big) \, x^{\scriptscriptstyle \frac{3}{4}}
\, + \; O_{\scriptscriptstyle \delta}\left(x^{\scriptscriptstyle \frac{2}{3} + \delta}\right)
\end{equation*} 
where $Q_{\scriptscriptstyle W}(u)$ is a computable degree six polynomial.
\end{thm} 

The polynomial 
$
Q_{\scriptscriptstyle W}(u)
$ 
can be easily computed from the principal part of $Z_{\scriptscriptstyle 0}(s)$ at $s = 1.$ 

\vskip5pt
We note that the restriction to positive fundamental discriminants divisible by $8$ (see also \cite{Sound} 
and \cite{Young}) is solely made for simplicity.

\vskip5pt
{\bf Relation to previous work.} \!Moments in families of $L$-functions are a topic of great interest in analytic number theory because of connections to the generalized Lindel\"{o}f hypothesis, various nonvanishing conjectures, etc. \!The third moment of quadratic $L$-functions has been studied by Soundararajan in \cite{Sound} and by the first author with Goldfeld and Hoffstein in \cite{DGH}. \!The best prior estimate is due to Young in \cite{Young}, who obtains a smoothed asymptotic formula with error $O(x^{\frac{3}{4}+\epsilon}).$ The secondary term of size $x^{\frac{3}{4}}$ was conjectured in \cite{DGH}, and verified by Zhang in \cite{Zha} under certain meromorphicity and polynomial growth assumptions, 
which we shall remove; Alderson and Rubinstein \cite{AR} have also given computational evidence for a secondary term.

The existence of this term raises many interesting questions. \!No analogous term exists in the asymptotics of the first two moments, \!and its existence does not seem to be predicted by random matrix-type models. \!Secondary terms of this type have yet to be fully incorporated into the framework of moment conjectures. \!Yet, \!work of both authors on Kac-Moody multiple Dirichlet series \cite{DV}, \cite{White1}, \cite{White2} predicts that many similar secondary terms will appear in higher moments of quadratic L-functions. \!One problem of interest to the authors concerns is the fourth moment of quadratic Dirichlet L-functions in the rational function field case, summed over monic square-free polynomials, see \cite{F}. Here the underlying group of symmetries is {\it infinite}, and the $p$-part $Z_{\scriptscriptstyle p}(s_{\scriptscriptstyle 1}, s_{\scriptscriptstyle 2}, s_{\scriptscriptstyle 3}, s_{\scriptscriptstyle 4}, s_{\scriptscriptstyle 5})$ is more challenging to understand, but we again expect a secondary term of size $x^{\frac{3}{4}}$ to exist. \!In fact, we expect infinitely many secondary terms of sizes between $x^{\frac{3}{4}}$ and  $x^{\frac{1}{2}}.$

The multiple Dirichlet series approach of \cite{DGH}, \cite{Zha}, and the present work explains the presence of secondary terms as follows. The series $Z_{\scriptscriptstyle 0}(s)$ is a specialization of a four-variable Dirichlet series
\begin{equation*} 
\sideset{}{^*}\sum_{d} L\big(s_{\scriptscriptstyle 1}, \chi_{\scriptscriptstyle 2 d} \big)L\big(s_{\scriptscriptstyle 2}, \chi_{\scriptscriptstyle 2 d} \big)L\big(s_{\scriptscriptstyle 3}, \chi_{\scriptscriptstyle 2 d} \big) d^{\scriptscriptstyle - s_{\scriptscriptstyle 4}}.
\end{equation*}
To take full advantage of symmetry, it is helpful to work with modifications of this object, denoted $Z(s_{\scriptscriptstyle 1}, s_{\scriptscriptstyle 2}, s_{\scriptscriptstyle 3}, s_{\scriptscriptstyle 4})$, where the sum is over all positive integers $d.$ 
When $d$ has square factors, the $L$-functions appearing in the sum are altered at finitely many primes. The resulting object satisfies a group of functional equations isomorphic to the Weyl group of the root system $D_{4}.$ \!These functional equations imply meromorphic continuation to all of $\mathbb{C}^{4}$ via an application of Bochner's principle. Furthermore, they fully determine the polar divisors of $Z(s_{\scriptscriptstyle 1}, s_{\scriptscriptstyle 2}, s_{\scriptscriptstyle 3}, s_{\scriptscriptstyle 4})$, which are in one-to-one correspondence with the positive roots of $D_4$. It follows that $Z(\frac{1}{2}, \frac{1}{2}, \frac{1}{2}, s)$ has a simple pole at $s=\frac{3}{4}.$

We establish that this pole remains when the sum is restricted to fundamental discriminants -- it is not merely an artifact of the modifications in the multiple Dirichlet series construction. The series $Z_{\scriptscriptstyle 0}(s)$ can be obtained from quadratic twists of $Z(\frac{1}{2}, \frac{1}{2}, \frac{1}{2}, s),$ see Section 3. Under the hypothesis that 
$Z_{\scriptscriptstyle 0}(s)$\footnote{More precisely, the corresponding series obtained by summing over all fundamental 
discriminants.} has meromorphic continuation and polynomial growth in a half-plane containing $\frac{3}{4},$ Zhang \cite{Zha} computes the secondary term. Although our calculation of the residue at $s = \\ \frac{3}{4}$ of $Z_{\scriptscriptstyle 0}(s)$ is quite different from Zhang's calculation, \!we reach, essentially, the same answer; \!the only differences\footnote{After all our average is different from that studied in \cite{Zha}.} occur in the constants related to the places $2$ and $\infty.$ In particular, our calculation of the residue confirms the fairly complicated product over odd primes 
found by Zhang.         

{\bf Overview of the argument.} \!In the present work, we omit all discussion of the principal part of 
$Z_{\scriptscriptstyle 0}(s)$ at $s = 1.$ As already noted, this can be computed as in \cite[Section~3.2]{DGH}, 
or as in \cite{Young}. The proof of Theorem \ref{Main Theorem A} proceeds as follows. Section \ref{Properties} defines the multiple Dirichlet series $Z^{(c)}\big(s_{\scriptscriptstyle 1}, s_{\scriptscriptstyle 2}, s_{\scriptscriptstyle 3}, s_{\scriptscriptstyle 4}; \chi_{a_{\scriptscriptstyle 2}c_{\scriptscriptstyle 2}}, \chi_{a_{\scriptscriptstyle 1}c_{\scriptscriptstyle 1}}\big),$ which has roughly the form
\begin{equation*}
Z^{(c)}\big(s_{\scriptscriptstyle 1}, s_{\scriptscriptstyle 2}, s_{\scriptscriptstyle 3}, s_{\scriptscriptstyle 4}; \chi_{a_{\scriptscriptstyle 2}c_{\scriptscriptstyle 2}}, \chi_{a_{\scriptscriptstyle 1}c_{\scriptscriptstyle 1}}\big) \; = \sum_{\substack{d \, > \, 0 \\ \gcd(d, 2c)=1}} L^{(2c)}(s_{\scriptscriptstyle 1}, \chi_{a_{\scriptscriptstyle 1}c_{\scriptscriptstyle 1}}\chi_{\scriptscriptstyle d})L^{(2c)}(s_{\scriptscriptstyle 2}, \chi_{a_{\scriptscriptstyle 1}c_{\scriptscriptstyle 1}}\chi_{\scriptscriptstyle d})L^{(2c)}(s_{\scriptscriptstyle 3}, \chi_{a_{\scriptscriptstyle 1}c_{\scriptscriptstyle 1}}\chi_{\scriptscriptstyle d}) \chi_{a_{\scriptscriptstyle 2}c_{\scriptscriptstyle 2}}(d) d^{-s_{\scriptscriptstyle 4}}
\end{equation*} 
for an odd square-free positive integer $c$ and quadratic characters $\chi_{a_{\scriptscriptstyle 1}c_{\scriptscriptstyle 1}}, \chi_{a_{\scriptscriptstyle 2}c_{\scriptscriptstyle 2}}$ of conductors dividing $8c.$ \!The Euler factors of $L$-functions appearing in this sum are modified at primes $p$ such that $p^{2}\mid d.$ We list the known properties of $Z^{(c)}\big(s_{\scriptscriptstyle 1}, s_{\scriptscriptstyle 2}, s_{\scriptscriptstyle 3}, s_{\scriptscriptstyle 4}; \chi_{a_{\scriptscriptstyle 2}c_{\scriptscriptstyle 2}}, \chi_{a_{\scriptscriptstyle 1}c_{\scriptscriptstyle 1}}\big),$ including meromorphic continuation, functional equations, and a convexity bound for its size at $s_{\scriptscriptstyle 1}=s_{\scriptscriptstyle 2}=s_{\scriptscriptstyle 3}=\frac{1}{2}.$ Finally, we compute the residue of this series at $s_{\scriptscriptstyle 1}=s_{\scriptscriptstyle 2}=s_{\scriptscriptstyle 3}=\frac{1}{2}$, $s_{\scriptscriptstyle 4}=\frac{3}{4}.$

Section \ref{Sieve} uses a simple sieve to express $Z_{\scriptscriptstyle 0}(s)$ in terms of $Z^{(c)}\big(s_{\scriptscriptstyle 1}, s_{\scriptscriptstyle 2}, s_{\scriptscriptstyle 3}, s_{\scriptscriptstyle 4}; \chi_{a_{\scriptscriptstyle 2}c_{\scriptscriptstyle 2}}, \chi_{a_{\scriptscriptstyle 1}c_{\scriptscriptstyle 1}}\big).$ The main problem we face is to establish enough analytic continuation for the sieving formula. The crucial ingredient is an improvement in the bound for $Z^{(c)}\big(s_{\scriptscriptstyle 1}, s_{\scriptscriptstyle 2}, s_{\scriptscriptstyle 3}, s_{\scriptscriptstyle 4}; 
\chi_{a_{\scriptscriptstyle 2}c_{\scriptscriptstyle 2}}, \chi_{a_{\scriptscriptstyle 1}c_{\scriptscriptstyle 1}}\big)$ at $s_{\scriptscriptstyle 1}=s_{\scriptscriptstyle 2}=s_{\scriptscriptstyle 3}=\frac{1}{2},$ with $s_{\scriptscriptstyle 4}=s$ fixed. 
If $c=c_{\scriptscriptstyle 1} c_{\scriptscriptstyle 2} c_{\scriptscriptstyle 3},$ the convexity bound is not enough in $c_{\scriptscriptstyle 3}$ aspect to imply that the sieving formula continues beyond $\Re(s) = \frac{3}{4}.$ However, 
the recursive refinement in Proposition \ref{key-proposition} provides substantial improvement of the exponent of $c_{\scriptscriptstyle 3},$ allowing us to establish the desired analytic continuation.

Section \ref{Proofs} completes the proofs of our main theorems. \!After multiplying by a polynomial to remove the poles at $s=1$ and $s=\frac{3}{4},$ we show that the sieving formula converges absolutely for $\Re(s)>\frac{2}{3}.$ The residue of $Z_{\scriptscriptstyle 0}(s)$ at $s=\frac{3}{4}$ is then computed. Theorem \ref{Main Theorem B} is deduced via a contour integration.

The recursive refinement argument of Proposition \ref{key-proposition} is a new technique which has not appeared elsewhere in the multiple Dirichlet series literature. \!This technique is now available because of the extent to which local factors, \!or $p$-parts, of multiple Dirichlet series are understood. The $p$-part of $Z,$ denoted $Z_{\scriptscriptstyle p}(s_{\scriptscriptstyle 1}, s_{\scriptscriptstyle 2}, s_{\scriptscriptstyle 3}, s_{\scriptscriptstyle 4}),$ is a power series in 
$p^{-s_i}$ which serves as a generating function for the modified Euler factors appearing in the multiple Dirichlet series. There is an extensive literature on such $p$-parts, which are of interest for their connection to local representations of metaplectic groups -- see, for example, \cite{BBF2, CG2, McN2}. In other articles, the authors have proposed an axiomatic characterization of the $p$-parts which is designed to extend to infinite-dimensional groups \cite{DV, White1}. Of central importance is a dominance axiom, which bounds the coefficients of $Z_{\scriptscriptstyle p}(s_{\scriptscriptstyle 1}, s_{\scriptscriptstyle 2}, s_{\scriptscriptstyle 3}, s_{\scriptscriptstyle 4})$ by powers of $p.$ One consequence is that, aside from fixed constant and linear terms, the expressions $Z_{\scriptscriptstyle p}(s_{\scriptscriptstyle 1}, s_{\scriptscriptstyle 2}, s_{\scriptscriptstyle 3}, s_{\scriptscriptstyle 4})$ have decay in $p$ for $\Re(s_i)>\frac{1}{2}.$ \!The recursive refinement method here plays convexity bounds for $Z^{(c)}\big(s_{\scriptscriptstyle 1}, s_{\scriptscriptstyle 2}, s_{\scriptscriptstyle 3}, s_{\scriptscriptstyle 4}; \chi_{a_{\scriptscriptstyle 2}c_{\scriptscriptstyle 2}}, \chi_{a_{\scriptscriptstyle 1}c_{\scriptscriptstyle 1}}\big)$ against explicit bounds for its $p$-parts. \!The axiomatic approach is not strictly necessary, 
because in this case we have 
$Z_{\scriptscriptstyle p}(s_{\scriptscriptstyle 1}, s_{\scriptscriptstyle 2}, s_{\scriptscriptstyle 3}, s_{\scriptscriptstyle 4})$ as an explicit rational function, but it certainly informs the technique.

\vskip10pt
\begin{rem} \!\!The $D_{4}$ Weyl group multiple Dirichlet series 
$
Z^{(c)}\big(s_{\scriptscriptstyle 1}, s_{\scriptscriptstyle 2}, s_{\scriptscriptstyle 3}, s_{\scriptscriptstyle 4}; \chi_{a_{\scriptscriptstyle 2}c_{\scriptscriptstyle 2}}, \chi_{a_{\scriptscriptstyle 1}c_{\scriptscriptstyle 1}}\big)
$ 
can be generalized to arbitrary number fields, see, for instance, \cite{CG1}. They possess similar analytic properties (e.g., 
meromorphic continuation to $\mathbb{C}^{4},$ polynomial growth), and can be used to establish asymptotics 
for the cubic moment of quadratic L-series over number fields. It is worth noticing the presence of the central 
value $\zeta_{\scriptstyle K}(\frac{1}{2})^{7}$ of the Dedekind zeta-function of a number field $K$ in the constant 
of the $x^{\frac{3}{4}}$-term in the corresponding asymptotic formula for the cubic moment over $K.$ This phenomenon 
and its potential relevance remains to be further investigated.    
\end{rem}

\vskip1pt
We designate this article as a sequel to \cite{D} to emphasize the similarity between the function field and number field cases. \!The structure of the rational function field proof given in \cite{D} is parallel to that of the present work. \!We intend the two articles together to serve as a model for transferring multiple Dirichlet series arguments from the geometric to the arithmetic setting.

\section{Properties of Multiple {D}irichlet series} \label{Properties} 
{\bf Definitions.} For $d\in \ZZ$ non-zero and square-free, let $\chi_{d}(m)$ be the quadratic character defined by 
\begin{equation*} 
\chi_{d}(m) = 
\begin{cases} 
\left(\frac{d}{m}\right) & \mbox{if $d \equiv 1 \!\!\!\!\!\pmod 4$} \\
\left(\frac{4d}{m}\right) & \mbox{if $d \equiv 2, \, 3 \!\!\!\!\!\pmod 4.$}
\end{cases}  
\end{equation*} 
Fix an odd, positive, square-free integer $c.$ Let $a_{\scriptscriptstyle 1}, a_{\scriptscriptstyle 2} \in \{\pm 1, \, \pm2\},$ 
and let $c_{\scriptscriptstyle 1}, c_{\scriptscriptstyle 2}$ divide $c.$ We will study the multiple Dirichlet series \begin{equation*} 
Z^{(c)}\big(s_{\scriptscriptstyle 1}, s_{\scriptscriptstyle 2}, s_{\scriptscriptstyle 3}, s_{\scriptscriptstyle 4}; \chi_{a_{\scriptscriptstyle 2}c_{\scriptscriptstyle 2}}, \chi_{a_{\scriptscriptstyle 1}c_{\scriptscriptstyle 1}}\big) \; =\sum_{\substack{m_{\scriptscriptstyle 1}, \, m_{\scriptscriptstyle 2}, \, m_{\scriptscriptstyle 3},
\, d \, \geq \,1 \\ \gcd(m_{\scriptscriptstyle 1}m_{\scriptscriptstyle 2}m_{\scriptscriptstyle 3}d, 2c)=1}} 
H(m_{\scriptscriptstyle 1}, m_{\scriptscriptstyle 2}, m_{\scriptscriptstyle 3}, d)\chi_{a_{\scriptscriptstyle 1}c_{\scriptscriptstyle 1}}(m_{\scriptscriptstyle 1}m_{\scriptscriptstyle 2}m_{\scriptscriptstyle 3}) 
\chi_{a_{\scriptscriptstyle 2}c_{\scriptscriptstyle 2}}(d) m_{\scriptscriptstyle 1}^{-s_{\scriptscriptstyle 1}}
m_{\scriptscriptstyle 2}^{-s_{\scriptscriptstyle 2}}m_{\scriptscriptstyle 3}^{-s_{\scriptscriptstyle 3}}
d^{-s_{\scriptscriptstyle 4}}.  
\end{equation*} 
The function $H(m_{\scriptscriptstyle 1}, m_{\scriptscriptstyle 2}, m_{\scriptscriptstyle 3}, d)$ on quadruples of odd 
integers is defined as follows. \!First, $H$ satisfies a {\it twisted} multiplicativity property. For 
$\gcd(m_{\scriptscriptstyle 1}^{} m_{\scriptscriptstyle 2}^{} m_{\scriptscriptstyle 3}^{} d, 
m_{\scriptscriptstyle 1}' m_{\scriptscriptstyle 2}' m_{\scriptscriptstyle 3}' d') = 1,$ we have\footnote{Note that in the conventions of \cite{CG1}, which are largely followed here, the Kronecker symbols would be flipped. \!However, this convention is more convenient for working with the family $L(s, \chi_d).$}: 
\begin{equation*}
H(m_{\scriptscriptstyle 1}^{} m_{\scriptscriptstyle 1}', m_{\scriptscriptstyle 2}^{} m_{\scriptscriptstyle 2}', 
m_{\scriptscriptstyle 3}^{} m_{\scriptscriptstyle 3}', dd') = H(m_{\scriptscriptstyle 1}^{}, m_{\scriptscriptstyle 2}^{}, m_{\scriptscriptstyle 3}^{}, d) H(m_{\scriptscriptstyle 1}', m_{\scriptscriptstyle 2}', m_{\scriptscriptstyle 3}', d')
\left(\frac{d}{m_{\scriptscriptstyle 1}' m_{\scriptscriptstyle 2}' m_{\scriptscriptstyle 3}'}\right)
\left(\frac{d'}{m_{\scriptscriptstyle 1}^{} m_{\scriptscriptstyle 2}^{} m_{\scriptscriptstyle 3}^{}}\right).
\end{equation*}
Given this property, it suffices to define $H(p^{k_{\scriptscriptstyle 1}}, p^{k_{\scriptscriptstyle 2}}, 
p^{k_{\scriptscriptstyle 3}}, p^{l})$ for $p$ prime. \!These coefficients are given by an explicit generating function
\begin{equation*} 
Z_{\scriptscriptstyle p}(s_{\scriptscriptstyle 1}, s_{\scriptscriptstyle 2}, s_{\scriptscriptstyle 3}, s_{\scriptscriptstyle 4}) \; =\sum_{k_{\scriptscriptstyle 1}, \, k_{\scriptscriptstyle 2}, \, k_{\scriptscriptstyle 3}, \, l \, \geq \, 0} \, 
H(p^{k_{\scriptscriptstyle 1}}, p^{k_{\scriptscriptstyle 2}}, p^{k_{\scriptscriptstyle 3}}, p^{l}) p^{ - k_{\scriptscriptstyle 1}s_{\scriptscriptstyle 1} - k_{\scriptscriptstyle 2}s_{\scriptscriptstyle 2} - k_{\scriptscriptstyle 3}s_{\scriptscriptstyle 3} -ls_{\scriptscriptstyle 4}}
\end{equation*}
known as the $p$-part of the series. More precisely, 
\begin{equation*}
Z_{\scriptscriptstyle p}(s_{\scriptscriptstyle 1}, s_{\scriptscriptstyle 2}, s_{\scriptscriptstyle 3}, s_{\scriptscriptstyle 4}) 
= f\big(p^{- s_{\scriptscriptstyle 1}}\!, \, p^{- s_{\scriptscriptstyle 2}}\!, \, p^{- s_{\scriptscriptstyle 3}}\!, \,
p^{- s_{\scriptscriptstyle 4}}; p\big)
\end{equation*} 
where 
$ 
f(z_{\scriptscriptstyle 1}^{}\!, \, z_{\scriptscriptstyle 2}^{}, \, z_{\scriptscriptstyle 3}^{}, \, z_{\scriptscriptstyle 4}^{}; p) 
$ 
is the rational function given in \cite[Appendix B, Equation 32]{D}\footnote{Note that \cite{D} uses the notation $A(p^{k_{\scriptscriptstyle 1}}\!, p^{k_{\scriptscriptstyle 2}}\!, p^{k_{\scriptscriptstyle 3}}\!, p^{l})$ or 
$a(k_{\scriptscriptstyle 1}, k_{\scriptscriptstyle 2}, k_{\scriptscriptstyle 3}, l; p)$ for $H(p^{k_{\scriptscriptstyle 1}}\!, p^{k_{\scriptscriptstyle 2}}\!, p^{k_{\scriptscriptstyle 3}}\!, p^{l}).$}\!\!.

From the generating series $Z_{\scriptscriptstyle p},$ we have that $H(p^{k_{\scriptscriptstyle 1}}\!, p^{k_{\scriptscriptstyle 2}}\!, p^{k_{\scriptscriptstyle 3}}\!, p^{l})=1$ when $\min(k_{\scriptscriptstyle 1}+k_{\scriptscriptstyle 2}+k_{\scriptscriptstyle 3}, l)=0$ and $H(p^{k_{\scriptscriptstyle 1}}\!, p^{k_{\scriptscriptstyle 2}}\!, p^{k_{\scriptscriptstyle 3}}\!, p^{l})\\ =0$ when 
$\min(k_{\scriptscriptstyle 1}+k_{\scriptscriptstyle 2}+k_{\scriptscriptstyle 3}, l)=1.$ Therefore, 
$H(m_{\scriptscriptstyle 1}, m_{\scriptscriptstyle 2}, m_{\scriptscriptstyle 3}, d) = 
\big(\frac{d}{m_{\scriptscriptstyle 1}m_{\scriptscriptstyle 2}m_{\scriptscriptstyle 3}}\big)$ whenever either $d$ or $m_{\scriptscriptstyle 1}m_{\scriptscriptstyle 2}m_{\scriptscriptstyle 3}$ is square-free. Furthermore, $H(p^{k_{\scriptscriptstyle 1}}\!, p^{k_{\scriptscriptstyle 2}}\!, p^{k_{\scriptscriptstyle 3}}\!, p^{l})=0$ when $k_{\scriptscriptstyle 1}+k_{\scriptscriptstyle 2}+k_{\scriptscriptstyle 3}$ and $l$ are both odd. We can compare 
$Z^{(c)}(s_{\scriptscriptstyle 1}, s_{\scriptscriptstyle 2}, s_{\scriptscriptstyle 3}, s_{\scriptscriptstyle 4}; \chi_{a_{\scriptscriptstyle 2}c_{\scriptscriptstyle 2}}, \chi_{a_{\scriptscriptstyle 1}c_{\scriptscriptstyle 1}})$ 
to an Euler product of $Z_{\scriptscriptstyle p}$ factors to show that it converges absolutely for 
$\Re(s_{\scriptscriptstyle 1}), \ldots, \Re(s_{\scriptscriptstyle 4})\, > \, 1.$

We will rewrite the function 
$
Z^{(c)}\big(s_{\scriptscriptstyle 1}, s_{\scriptscriptstyle 2}, s_{\scriptscriptstyle 3}, s_{\scriptscriptstyle 4}; \chi_{a_{\scriptscriptstyle 2}c_{\scriptscriptstyle 2}}, \chi_{a_{\scriptscriptstyle 1}c_{\scriptscriptstyle 1}}\big)
$ 
in two different ways which allow us to verify its meromorphic continuation to $\CC^{4}$ and the group of functional equations. \!Fix a positive integer $d,$ coprime to $2c,$ which factors as 
$d_{\scriptscriptstyle 0}^{}d_{\scriptscriptstyle 1}^{2}$ with $d_{\scriptscriptstyle 0}^{}$ square-free. \!Then we have 
\begin{align*}
& \sum_{\substack{m_{\scriptscriptstyle 1}, \, m_{\scriptscriptstyle 2}, \, m_{\scriptscriptstyle 3} \, \geq 1\,  \\ 
\gcd(m_{\scriptscriptstyle 1}m_{\scriptscriptstyle 2}m_{\scriptscriptstyle 3}, 2c) = 1}} 
H(m_{\scriptscriptstyle 1}, m_{\scriptscriptstyle 2}, m_{\scriptscriptstyle 3}, d) \chi_{a_{\scriptscriptstyle 1}c_{\scriptscriptstyle 1}}(m_{\scriptscriptstyle 1}m_{\scriptscriptstyle 2}m_{\scriptscriptstyle 3}) 
m_{\scriptscriptstyle 1}^{- s_{\scriptscriptstyle 1}} 
m_{\scriptscriptstyle 2}^{- s_{\scriptscriptstyle 2}} m_{\scriptscriptstyle 3}^{- s_{\scriptscriptstyle 3}} \\
& = L^{(2c)}(s_{\scriptscriptstyle 1}, \chi_{a_{\scriptscriptstyle 1}c_{\scriptscriptstyle 1}d_{\scriptscriptstyle 0}})
L^{(2c)}(s_{\scriptscriptstyle 2}, \chi_{a_{\scriptscriptstyle 1}c_{\scriptscriptstyle 1}d_{\scriptscriptstyle 0}})
L^{(2c)}(s_{\scriptscriptstyle 3}, \chi_{a_{\scriptscriptstyle 1}c_{\scriptscriptstyle 1}d_{\scriptscriptstyle 0}}) \\ 
&\cdot \prod_{\substack{p^{l} \parallel d \\ l \, \geq \, 2}} \frac{\sum_{k_{\scriptscriptstyle 1}, \, k_{\scriptscriptstyle 2}, \, k_{\scriptscriptstyle 3} \, \geq \, 0}\, H(p^{k_{\scriptscriptstyle 1}}\!, p^{k_{\scriptscriptstyle 2}}\!, 
p^{k_{\scriptscriptstyle 3}}\!, p^{l}) \chi_{a_{\scriptscriptstyle 1}c_{\scriptscriptstyle 1}}(p)^{k_{\scriptscriptstyle 1} + k_{\scriptscriptstyle 2} + k_{\scriptscriptstyle 3}} \Big(\frac{d p^{-l}}{p}\Big)^{k_{\scriptscriptstyle 1} 
+ k_{\scriptscriptstyle 2} + k_{\scriptscriptstyle 3}} p^{- k_{\scriptscriptstyle 1} s_{\scriptscriptstyle 1} - k_{\scriptscriptstyle 2} s_{\scriptscriptstyle 2} - k_{\scriptscriptstyle 3} s_{\scriptscriptstyle 3}}}
{L_{\scriptscriptstyle p}(s_{\scriptscriptstyle 1}, \chi_{a_{\scriptscriptstyle 1}c_{\scriptscriptstyle 1}d_{\scriptscriptstyle 0}})
L_{\scriptscriptstyle p}(s_{\scriptscriptstyle 2}, \chi_{a_{\scriptscriptstyle 1}c_{\scriptscriptstyle 1}d_{\scriptscriptstyle 0}})
L_{\scriptscriptstyle p}(s_{\scriptscriptstyle 3}, \chi_{a_{\scriptscriptstyle 1}c_{\scriptscriptstyle 1}d_{\scriptscriptstyle 0}})}.
\end{align*}
Here $L^{(2c)}(s_{i}, \chi_{a_{\scriptscriptstyle 1}c_{\scriptscriptstyle 1}d_{\scriptscriptstyle 0}})$ denotes the quadratic Dirichlet $L$-function with Euler factors at primes $p$ dividing $2c$ removed; $L_{\scriptscriptstyle p}(s_{i}, \chi_{a_{\scriptscriptstyle 1}c_{\scriptscriptstyle 1}d_{\scriptscriptstyle 0}})$ denotes the Euler factor at $p.$ The latter product is a Dirichlet polynomial we denote as 
$
P_{\scriptscriptstyle d}(s_{\scriptscriptstyle 1}, s_{\scriptscriptstyle 2}, s_{\scriptscriptstyle 3}; 
\chi_{a_{\scriptscriptstyle 1} c_{\scriptscriptstyle 1} d_{\scriptscriptstyle 0}}).
$ 

The local coefficients $H$ are so that this modified product of $L$-functions satisfies uniform functional equations. 
\!For $i = 1, 2, 3,$ and for $a=0$ if 
$\chi_{a_{\scriptscriptstyle 1}c_{\scriptscriptstyle 1}}(-1)=1,$ $a=1$ if $\chi_{a_{\scriptscriptstyle 1}c_{\scriptscriptstyle 1}}(-1) = -1,$ the function
\begin{equation*}
\left(\tfrac{\pi}{d_{\scriptscriptstyle 1}^{2}\,\cond\, \chi_{a_{\scriptscriptstyle 1}c_{\scriptscriptstyle 1}d_{\scriptscriptstyle 0}}}\right)^{\!\! - \frac{s_i + a}{2}}
\Gamma\!\left(\tfrac{s_i + a}{2}\right) 
L^{(2c)}(s_i, \chi_{a_{\scriptscriptstyle 1}c_{\scriptscriptstyle 1}d_{\scriptscriptstyle 0}})
\Bigg(\prod_{p \mid 2c} L_{\scriptscriptstyle p}(s_i, \chi_{a_{\scriptscriptstyle 1}c_{\scriptscriptstyle 1}d_{\scriptscriptstyle 0}})\Bigg) 
P_{\scriptscriptstyle d}(s_{\scriptscriptstyle 1}, s_{\scriptscriptstyle 2}, s_{\scriptscriptstyle 3}; 
\chi_{a_{\scriptscriptstyle 1} c_{\scriptscriptstyle 1} d_{\scriptscriptstyle 0}})
\end{equation*}
is symmetric under $s_i \mapsto 1 - s_i,$ for all $d.$

In the domain of absolute convergence, we have 
\begin{equation} \label{MDS-vers1}
\begin{split}
&Z^{(c)}\big(s_{\scriptscriptstyle 1}, s_{\scriptscriptstyle 2}, s_{\scriptscriptstyle 3}, s_{\scriptscriptstyle 4}; 
\chi_{a_{\scriptscriptstyle 2}c_{\scriptscriptstyle 2}}, \chi_{a_{\scriptscriptstyle 1}c_{\scriptscriptstyle 1}}\big)  \\
& = \sum_{\substack{d \, \geq \, 1 \\ \gcd(d, 2c) = 1}} \chi_{a_{\scriptscriptstyle 2}c_{\scriptscriptstyle 2}}(d) L^{(2c)}(s_{\scriptscriptstyle 1}, \chi_{a_{\scriptscriptstyle 1}c_{\scriptscriptstyle 1}d_{\scriptscriptstyle 0}})
L^{(2c)}(s_{\scriptscriptstyle 2}, \chi_{a_{\scriptscriptstyle 1}c_{\scriptscriptstyle 1}d_{\scriptscriptstyle 0}})
L^{(2c)}(s_{\scriptscriptstyle 3}, \chi_{a_{\scriptscriptstyle 1}c_{\scriptscriptstyle 1}d_{\scriptscriptstyle 0}})
P_{\scriptscriptstyle d}(s_{\scriptscriptstyle 1}, s_{\scriptscriptstyle 2}, s_{\scriptscriptstyle 3}; 
\chi_{a_{\scriptscriptstyle 1} c_{\scriptscriptstyle 1} d_{\scriptscriptstyle 0}}) d^{- s_{\scriptscriptstyle 4}}
\end{split}
\end{equation} 
where each $d$ is factored into $d_{\scriptscriptstyle 0}^{}d_{\scriptscriptstyle 1}^{2},$ as above. 
\!The polynomials $P_{\scriptscriptstyle d}(s_{\scriptscriptstyle 1}, s_{\scriptscriptstyle 2}, s_{\scriptscriptstyle 3}; 
\chi_{a_{\scriptscriptstyle 1} c_{\scriptscriptstyle 1} d_{\scriptscriptstyle 0}})$ have polynomial growth in $d_{\scriptscriptstyle 1}.$ 
\!It follows that the sum converges absolutely for every $s_{\scriptscriptstyle 1}, s_{\scriptscriptstyle 2}, s_{\scriptscriptstyle 3} \neq 1$ 
as long as $\Re(s_{\scriptscriptstyle 4})$ is sufficiently large.

The second expression for $Z^{(c)}\big(s_{\scriptscriptstyle 1}, s_{\scriptscriptstyle 2}, s_{\scriptscriptstyle 3}, s_{\scriptscriptstyle 4}; \chi_{a_{\scriptscriptstyle 2}c_{\scriptscriptstyle 2}}, \chi_{a_{\scriptscriptstyle 1}c_{\scriptscriptstyle 1}}\big)$ evaluates the $d$ and 
$m_{\scriptscriptstyle 1}, m_{\scriptscriptstyle 2}, m_{\scriptscriptstyle 3}$ sums in the opposite order. We will use the notation 
$\tilde{\chi}_{m}(d)$ for the character defined by the Kronecker symbol $\left(\frac{d}{m}\right).$ Fix positive integers 
$m_{\scriptscriptstyle 1}, m_{\scriptscriptstyle 2}, m_{\scriptscriptstyle 3}$ coprime to $2c,$ and let $m_{0}$ 
denote the square-free part of the product $m_{\scriptscriptstyle 1}m_{\scriptscriptstyle 2}m_{\scriptscriptstyle 3}.$ We have: 
\begin{align*}
& \sum_{\substack{d \, \geq \, 1 \\ \gcd(d, 2c) = 1}} 
H(m_{\scriptscriptstyle 1}, m_{\scriptscriptstyle 2}, m_{\scriptscriptstyle 3}, d) 
\chi_{a_{\scriptscriptstyle 2}c_{\scriptscriptstyle 2}}(d)d^{- s_{\scriptscriptstyle 4}} \\
& = L^{(2c)}(s_{\scriptscriptstyle 4}, \chi_{a_{\scriptscriptstyle 2}c_{\scriptscriptstyle 2}}\tilde{\chi}_{m_0}) 
\prod_{\substack{p^{k_{\scriptscriptstyle 1}}\parallel m_{\scriptscriptstyle 1}, \, 
p^{k_{\scriptscriptstyle 2}}\parallel m_{\scriptscriptstyle 2}, \, 
p^{k_{\scriptscriptstyle 3}}\parallel m_{\scriptscriptstyle 3}
\\ k_{\scriptscriptstyle 1} + k_{\scriptscriptstyle 2} + k_{\scriptscriptstyle 3} \, \geq \, 2}} 
\frac{\sum_{l \geq 0}\, H(p^{k_{\scriptscriptstyle 1}}, p^{k_{\scriptscriptstyle 2}}, p^{k_{\scriptscriptstyle 3}}, p^{l}) 
\chi_{a_{\scriptscriptstyle 2}c_{\scriptscriptstyle 2}}(p)^{l} 
\left(\frac{p}{m_{\scriptscriptstyle 1}m_{\scriptscriptstyle 2}m_{\scriptscriptstyle 3}p^{- k_{\scriptscriptstyle 1} - k_{\scriptscriptstyle 2} - 
k_{\scriptscriptstyle 3}}}\right)^{\! l} p^{- ls_{\scriptscriptstyle 4}}}
{L_{\scriptscriptstyle p}(s_{\scriptscriptstyle 4}, \chi_{a_{\scriptscriptstyle 2}c_{\scriptscriptstyle 2}}\tilde{\chi}_{m_{0}})}.
\end{align*} 
The latter product is a Dirichlet polynomial denoted 
$
Q_{m_{\scriptscriptstyle 1},\,  m_{\scriptscriptstyle 2}, \, m_{\scriptscriptstyle 3}}(s_{\scriptscriptstyle 4}; 
\chi_{a_{\scriptscriptstyle 2}c_{\scriptscriptstyle 2}}\tilde{\chi}_{m_{0}}).
$ 
This modified $L$-function also satisfies a uniform functional equation. Let $a = 0$ if 
$\chi_{a_{\scriptscriptstyle 2}c_{\scriptscriptstyle 2}}(-1)\tilde{\chi}_{m_{0}}(-1) = 1,$ and $a = 1$ if 
$\chi_{a_{\scriptscriptstyle 2}c_{\scriptscriptstyle 2}}(-1)\tilde{\chi}_{m_{0}}(-1) = -1.$ Then the function 
\begin{equation*}
\left(\tfrac{\pi}{m_{\scriptscriptstyle 1}m_{\scriptscriptstyle 2}m_{\scriptscriptstyle 3} \, 
\cond \, \chi_{a_{\scriptscriptstyle 2}c_{\scriptscriptstyle 2}}}\right)^{\!\! - \frac{s_{\scriptscriptstyle 4} + a}{2}}
\Gamma\!\left(\tfrac{s_{\scriptscriptstyle 4} + a}{2}\right) 
L^{(2c)}(s_{\scriptscriptstyle 4}, \chi_{a_{\scriptscriptstyle 2}c_{\scriptscriptstyle 2}}\tilde{\chi}_{m_{0}})
\Bigg(\prod_{p\mid 2c} L_{\scriptscriptstyle p}(s_{\scriptscriptstyle 4}, \chi_{a_{\scriptscriptstyle 2}c_{\scriptscriptstyle 2}}\tilde{\chi}_{m_{0}}) \Bigg) 
Q_{m_{\scriptscriptstyle 1}, \, m_{\scriptscriptstyle 2}, \, m_{\scriptscriptstyle 3}}(s_{\scriptscriptstyle 4}; 
\chi_{a_{\scriptscriptstyle 2}c_{\scriptscriptstyle 2}}\tilde{\chi}_{m_{0}})
\end{equation*} 
is symmetric under $s_{\scriptscriptstyle 4} \mapsto 1 - s_{\scriptscriptstyle 4},$ even when 
$m_{\scriptscriptstyle 1}m_{\scriptscriptstyle 2}m_{\scriptscriptstyle 3}$ is {\it not} square-free.

Thus we can write 
\begin{equation} \label{MDS-vers2}
Z^{(c)}\big(s_{\scriptscriptstyle 1}, s_{\scriptscriptstyle 2}, s_{\scriptscriptstyle 3}, s_{\scriptscriptstyle 4}; 
\chi_{a_{\scriptscriptstyle 2}c_{\scriptscriptstyle 2}}, \chi_{a_{\scriptscriptstyle 1}c_{\scriptscriptstyle 1}}\big) \;  
= \sum_{\substack{m_{\scriptscriptstyle 1},\, m_{\scriptscriptstyle 2},\, m_{\scriptscriptstyle 3} \, \geq \, 1 \\ 
\gcd(m_{\scriptscriptstyle 1}m_{\scriptscriptstyle 2}m_{\scriptscriptstyle 3}, 2c) = 1}} 
\frac{\chi_{a_{\scriptscriptstyle 1}c_{\scriptscriptstyle 1}}(m_{\scriptscriptstyle 1}m_{\scriptscriptstyle 2}m_{\scriptscriptstyle 3}) 
L^{(2c)}(s_{\scriptscriptstyle 4}, \chi_{a_{\scriptscriptstyle 2}c_{\scriptscriptstyle 2}}\tilde{\chi}_{m_{0}})
Q_{\scriptscriptstyle m_{\scriptscriptstyle 1},\, m_{\scriptscriptstyle 2}, \, m_{\scriptscriptstyle 3}}(s_{\scriptscriptstyle 4}; 
\chi_{a_{\scriptscriptstyle 2}c_{\scriptscriptstyle 2}}\tilde{\chi}_{m_{0}})}
{m_{\scriptscriptstyle 1}^{s_{\scriptscriptstyle 1}} 
m_{\scriptscriptstyle 2}^{s_{\scriptscriptstyle 2}}
m_{\scriptscriptstyle 3}^{s_{\scriptscriptstyle 3}}}.
\end{equation} 
As before, for any $s_{\scriptscriptstyle 4} \neq 1,$ the sum converges absolutely for $\Re(s_i)$ ($i = 1, 2, 3$) sufficiently large.

\vskip5pt
{\bf Functional equations and analytic continuation.} As shown in \cite{DGH}, the family of multiple Dirichlet series defined at the beginning of this section 
satisfies a group of functional equations. However, for the computation of the residue we are interested in, it is more convenient to write the functional 
equations as follows.   

For an arbitrary (primitive) quadratic Dirichlet character $\chi,$ let 
\begin{equation*}
\Lambda_{c}(s; \chi) = \left(\tfrac{\pi}{\gcd(\cond \chi, 8c)}\right)^{\!\! - \frac{s + a}{2}}\Gamma\!\left(\tfrac{s + a}{2}\right) 
\prod_{p \mid 2c} L_{\scriptscriptstyle p}(s, \chi) 
\end{equation*} 
where $a$ is $0$ if $\chi(-1) = 1$ and $1$ if $\chi(-1) = - 1.$

Let $\omega(c)$ denote the number of distinct prime factors of $c.$ For $D \in \ZZ$ coprime to $2c,$ the linear combination 
\begin{equation}  \label{lin-comb-modsq8c}
2^{- \omega(c) - 2}\sum_{\substack{a_{\scriptscriptstyle 2} \in \{\pm 1, \, \pm 2 \} \\ 
c_{\scriptscriptstyle 2}\mid c}} \chi_{a_{\scriptscriptstyle 2}c_{\scriptscriptstyle 2}}(D) 
Z^{(c)}\big(s_{\scriptscriptstyle 1}, s_{\scriptscriptstyle 2}, s_{\scriptscriptstyle 3}, s_{\scriptscriptstyle 4}; 
\chi_{a_{\scriptscriptstyle 2}c_{\scriptscriptstyle 2}}, \chi_{a_{\scriptscriptstyle 1}c_{\scriptscriptstyle 1}}\big) 
\end{equation} 
isolates the summands of $d$ in 
$Z^{(c)}\big(s_{\scriptscriptstyle 1}, s_{\scriptscriptstyle 2}, s_{\scriptscriptstyle 3}, s_{\scriptscriptstyle 4}; \chi_{a_{\scriptscriptstyle 2}c_{\scriptscriptstyle 2}}, \chi_{a_{\scriptscriptstyle 1}c_{\scriptscriptstyle 1}}\big)$ with $dD$ congruent to a square modulo $8c.$ 
\!We take $D$ square-free and ranging over a complete set of representatives for 
$
\frac{(\mathbb{Z}\slash 8c\mathbb{Z})^{*}}{(\mathbb{Z}\slash 8c\mathbb{Z})^{*2}}.
$ 
If the expression \eqref{lin-comb-modsq8c} is written in the form of \eqref{MDS-vers1}, it can be seen to satisfy a functional equation: 
\!for $i = 1, 2, 3,$ the function 
\begin{equation} \label{fe-1}
\Lambda_{c}(s_i; \chi_{a_{\scriptscriptstyle 1}c_{\scriptscriptstyle 1}D}) 2^{- \omega(c) - 2}
\sum_{\substack{a_{\scriptscriptstyle 2} \in \{\pm 1, \, \pm 2 \} \\ c_{\scriptscriptstyle 2}\mid c}} 
\chi_{a_{\scriptscriptstyle 2}c_{\scriptscriptstyle 2}}(D) Z^{(c)}\big(s_{\scriptscriptstyle 1}, s_{\scriptscriptstyle 2}, s_{\scriptscriptstyle 3}, s_{\scriptscriptstyle 4}; \chi_{a_{\scriptscriptstyle 2}c_{\scriptscriptstyle 2}}, \chi_{a_{\scriptscriptstyle 1}c_{\scriptscriptstyle 1}}\big) 
\end{equation}
is symmetric under the transformation $\sigma_i$ which takes $s_i$ to $1 - s_{i},$ $s_{\scriptscriptstyle 4}$ to $s_{\scriptscriptstyle 4} + s_i - \frac{1}{2},$ 
and fixes the other variables.

Similarly, we may use the expression \eqref{MDS-vers2} to deduce an additional functional equation. For $M\in \ZZ$ coprime to $2c,$ the function 
\begin{equation} \label{fe-2}
\Lambda_{c}(s_{\scriptscriptstyle 4}; \chi_{a_{\scriptscriptstyle 2}c_{\scriptscriptstyle 2}}\tilde{\chi}_M) 2^{- \omega(c) - 2}
\sum_{\substack{a_{\scriptscriptstyle 1} \in \{\pm 1, \, \pm 2 \} \\ c_{\scriptscriptstyle 1}\mid c}} 
\chi_{a_{\scriptscriptstyle 1}c_{\scriptscriptstyle 1}}(M) Z^{(c)}\big(s_{\scriptscriptstyle 1}, s_{\scriptscriptstyle 2}, s_{\scriptscriptstyle 3}, s_{\scriptscriptstyle 4}; \chi_{a_{\scriptscriptstyle 2}c_{\scriptscriptstyle 2}}, \chi_{a_{\scriptscriptstyle 1}c_{\scriptscriptstyle 1}}\big) 
\end{equation}
is symmetric under the transformation $\sigma_4(s_{\scriptscriptstyle 1}, s_{\scriptscriptstyle 2}, s_{\scriptscriptstyle 3}, s_{\scriptscriptstyle 4}) 
= \big(s_{\scriptscriptstyle 1} + s_{\scriptscriptstyle 4} - \frac{1}{2}, s_{\scriptscriptstyle 2} + s_{\scriptscriptstyle 4} - \frac{1}{2}, 
s_{\scriptscriptstyle 3} + s_{\scriptscriptstyle 4} - \frac{1}{2}, 1 - s_{\scriptscriptstyle 4}\big).$ 

These symmetries may be considered as vector functional equations for the collection of all $Z^{(c)}\big(s_{\scriptscriptstyle 1}, s_{\scriptscriptstyle 2}, s_{\scriptscriptstyle 3}, s_{\scriptscriptstyle 4}; \chi_{a_{\scriptscriptstyle 2}c_{\scriptscriptstyle 2}}, \chi_{a_{\scriptscriptstyle 1}c_{\scriptscriptstyle 1}}\big)$ 
when $c$ is fixed but $a_{\scriptscriptstyle 1}, a_{\scriptscriptstyle 2}, c_{\scriptscriptstyle 1}, c_{\scriptscriptstyle 2}$ are allowed to vary. \!The underlying transformations $\sigma_i$ generate a symmetry group isomorphic to the Weyl group of root system $D_4.$ Applying these symmetries to the initial region of meromorphicity for $Z^{(c)}\big(s_{\scriptscriptstyle 1}, s_{\scriptscriptstyle 2}, s_{\scriptscriptstyle 3}, s_{\scriptscriptstyle 4}; \chi_{a_{\scriptscriptstyle 2}c_{\scriptscriptstyle 2}}, \chi_{a_{\scriptscriptstyle 1}c_{\scriptscriptstyle 1}}\big)$ produces a collection of overlapping regions, the complement of a 
bounded set in $\CC^{4}.$ Bochner's principle \cite{Boh} then yields meromorphic continuation to all of $\CC^{4};$ \!this argument is carried out in detail in \cite{DGH}.

We remark that it actually suffices to work with smaller sums than those appearing in equations \eqref{fe-1} and \eqref{fe-2}. \!Since 
$\Lambda_{c}(s_i; \chi_{a_{\scriptscriptstyle 1}c_{\scriptscriptstyle 1}D})$ does not contain Euler factors at primes $p$ dividing $c_{\scriptscriptstyle 1},$ 
it suffices to sum over $c_{\scriptscriptstyle 2}$ dividing $c\slash c_{\scriptscriptstyle 1}$ in \eqref{fe-1}; this isolates summands of \eqref{MDS-vers1} 
with the same functional equations. \!Similarly, it suffices to sum over $c_{\scriptscriptstyle 1}$ dividing $c\slash c_{\scriptscriptstyle 2}$ in \eqref{fe-2}. 
In this way one may always work with multiple Dirichlet series $Z^{(c)}\big(s_{\scriptscriptstyle 1}, s_{\scriptscriptstyle 2}, 
s_{\scriptscriptstyle 3}, s_{\scriptscriptstyle 4}; \chi_{a_{\scriptscriptstyle 2}c_{\scriptscriptstyle 2}}, \chi_{a_{\scriptscriptstyle 1}c_{\scriptscriptstyle 1}}\big)$ 
for which $c_{1},$ $c_{2}$ are relatively prime. \!This is the convention of \cite{DGH} and is used in the proof of Proposition \ref{convexity-bound} 
(their Proposition 4.12). However, we find it convenient to work with the larger sums in computing the residue in Proposition \ref{residue} below.

\vskip5pt
{\bf Convexity bound.} \!The function 
$ 
Z^{(c)}\big(\tfrac{1}{2}, \tfrac{1}{2}, \tfrac{1}{2}, s; \chi_{a_{\scriptscriptstyle 2}c_{\scriptscriptstyle 2}}, \chi_{a_{\scriptscriptstyle 1}c_{\scriptscriptstyle 1}}\big) 
$ 
also satisfies a convexity bound, \!which we shall recall briefly. \!For details, we
refer to \cite[Proposition~4.12]{DGH}.

\vskip10pt
\begin{prop} \label{convexity-bound}
Suppose that $c=c_{\scriptscriptstyle 1}c_{\scriptscriptstyle 2}c_{\scriptscriptstyle 3}$ is square-free. \!Then for every $\delta > 0$ and 
$a_{\scriptscriptstyle 1}, a_{\scriptscriptstyle 2} \in$ $\{\pm 1, \pm 2\},$ we have the estimate
\begin{equation} \label{eq: basic-initial-estimate} 
\frac{(s - 1)^{\scriptscriptstyle 7} \big(s - \frac{3}{4}\big)}{(s + 1)^{\scriptscriptstyle 8}} \cdot
Z^{(c)}\big(\tfrac{1}{2}, \tfrac{1}{2}, \tfrac{1}{2}, s; 
\chi_{a_{\scriptscriptstyle 2}c_{\scriptscriptstyle 2}}, 
\chi_{a_{\scriptscriptstyle 1}c_{\scriptscriptstyle 1}}\big) \, \ll_{\scriptscriptstyle \delta} \, 
(1+ |s|)^{\scriptscriptstyle 5(1 - \Re(s)) \, + \, \delta} 
A_{0}^{\omega(c)} S(c, \delta)
(c_{\scriptscriptstyle 1} c_{\scriptscriptstyle 3})^{\scriptscriptstyle 3 (1 - \Re(s))} 
\, c_{\scriptscriptstyle 2}^{\scriptscriptstyle \frac{5}{2}(1 - \Re(s))}
c^{\scriptscriptstyle \delta}
\end{equation} 
for all $s$ with $0 \le \Re(s) \le 1.$ Here $A_{0}$ is some computable positive constant, and  
\begin{equation*} 
S(c, \delta) \; = 
\sum_{a \, = \, \pm 1, \, \pm 2}\;\, \sum_{b \, \mid \, c} 
\; \sum_{(d_{\scriptscriptstyle 0}, \, 2) \, = \, 1} 
\, \big|L^{\scriptscriptstyle (2)}\big(\tfrac{1}{2},
\chi_{a b d_{\scriptscriptstyle 0}} \big)\big|^{3} d_{\scriptscriptstyle 0}^{\scriptscriptstyle - 1 - (\delta \slash 30)}.
\end{equation*} 
\end{prop} 

Note that the characters 
$
\chi_{a b d_{\scriptscriptstyle 0}}(n): = \Big(\!\frac{a b d_{\scriptscriptstyle 0}}{n}\!\Big),
$ 
for odd positive $n,$ appearing in $S(c, \delta)$ may be imprimitive.

\begin{proof}
First, by \cite[Proposition~B.1]{D} (taking also into account the local parts at $3$), one finds that 
\begin{equation*}
\big|P_{\scriptscriptstyle d}\big(\tfrac{1}{2}, \tfrac{1}{2}, \tfrac{1}{2}; \chi_{a_{\scriptscriptstyle 1}c_{\scriptscriptstyle 1}d_{\scriptscriptstyle 0}}\big)\big| \, \le  \, 
\big(\! \tfrac{10084}{1 - 3^{\scriptscriptstyle - 2\,  \eta}}\!\big)^{\scriptscriptstyle \omega(d_{\scriptscriptstyle 1})} 
d_{\scriptscriptstyle 1}^{\scriptscriptstyle \frac{1}{2} + \eta} 
\end{equation*} 
for every small positive $\eta.$ Choosing $\eta = 1\slash 5,$ and letting $\Re(s) > 1,$ we have
\begin{equation*}  
\begin{split}
\big|Z^{(c)}\big(\tfrac{1}{2}, \tfrac{1}{2}, \tfrac{1}{2}, s; 
\chi_{a_{\scriptscriptstyle 2}c_{\scriptscriptstyle 2}}, 
\chi_{a_{\scriptscriptstyle 1}c_{\scriptscriptstyle 1}} \big)\big| 
\; &\le \sum_{\substack{(d, \, 2 c) \, = \, 1 \\  d = d_{\scriptscriptstyle 0}^{} d_{\scriptscriptstyle 1}^{2}}}
\frac{\big|L^{\scriptscriptstyle (2 c_{\scriptscriptstyle 2}c_{\scriptscriptstyle 3})}
\big(\tfrac{1}{2}, \chi_{a_{\scriptscriptstyle 1}c_{\scriptscriptstyle 1}d_{\scriptscriptstyle 0}}\big)\big|^{3}\,  
\big|P_{\scriptscriptstyle d}\big(\tfrac{1}{2}, \tfrac{1}{2}, \tfrac{1}{2}; \chi_{a_{\scriptscriptstyle 1}c_{\scriptscriptstyle 1}d_{\scriptscriptstyle 0}}\big)\big|}{d^{\Re(s)}}\\
& < \, 8^{\omega(c_{\scriptscriptstyle 2}c_{\scriptscriptstyle 3})} \; \cdot 
\sum_{(d_{\scriptscriptstyle 0}, \, 2 c) \, =  \, 1}
\frac{\big|L^{\scriptscriptstyle (2)}\big(\tfrac{1}{2}, \chi_{a_{\scriptscriptstyle 1}c_{\scriptscriptstyle 1}d_{\scriptscriptstyle 0}}\big)\big|^{3}}{d_{0}^{\Re(s)}} \!\sum_{d_{\scriptscriptstyle 1} \ge 1}  
\frac{(28358)^{\scriptscriptstyle \omega(d_{\scriptscriptstyle 1})}}
{d_{1}^{\scriptscriptstyle 13\slash 10}}\\ 
& \ll \, 8^{\omega(c_{\scriptscriptstyle 2}c_{\scriptscriptstyle 3})} \; \cdot 
\sum_{(d_{\scriptscriptstyle 0}, \, 2) \, =  \, 1}
\frac{\big|L^{\scriptscriptstyle (2)}\big(\tfrac{1}{2}, \chi_{a_{\scriptscriptstyle 1}c_{\scriptscriptstyle 1}d_{\scriptscriptstyle 0}}\big)\big|^{3}}{d_{0}^{\Re(s)}}.
\end{split}
\end{equation*} 
The last series is easily seen to be convergent by applying the Cauchy-Schwarz inequality and a well-known result 
of Heath-Brown \cite{H-B}. Applying the functional equations and the Phragmen-Lindel\"of principle, we obtain the result.
\!\!\footnote{In \cite[Proposition~4.12]{DGH}, the exponent of $c_{\scriptscriptstyle 2}$ in the convexity 
bound was chosen just for uniformity to be $3 (1 - \Re(s)).$ The better exponent with $3$ replaced by $\frac{5}{2}$ 
(which did not help improving the main results in loc. cit.) can be explained as follows. In the proof of \cite[Proposition~4.12]{DGH} (and in the notation therein), the power of any prime factor of $d_{1}^{} \! = l_{2}^{}$ 
dividing the expression 
$
(d_{1}^{} d_{2}^{}) (d_{2}^{} d_{3}^{}) (d_{3}^{} d_{4}^{})\cdots l_{4}^{ - 3} d_{4}^{}
$ 
simply cannot exceed $5.$} \end{proof}

As in \cite{D}, one of the main ingredients in the proof of Theorem \ref{Main Theorem A} is an improvement of 
\eqref{eq: basic-initial-estimate} in the $c_{\scriptscriptstyle 3}$-aspect. \!This will be established in Proposition 
\ref{key-proposition}.

\vskip5pt
{\bf Poles of multiple {D}irichlet series and their residues.} \!This section computes two residues of 
$
Z^{(c)}\!\big(s_{\scriptscriptstyle 1}, \ldots, s_{\scriptscriptstyle 4}; \chi_{a_{\scriptscriptstyle 2} c_{\scriptscriptstyle 2}}, 
\chi_{a_{\scriptscriptstyle 1} c_{\scriptscriptstyle 1}}\!\big)\!.
$ 
It follows from the functional equations that this expression has 12 possible polar hyperplanes corresponding to the positive roots of $D_{4}:$ 
$s_{\scriptscriptstyle 1},\, s_{\scriptscriptstyle 2},\, s_{\scriptscriptstyle 3} = 1,$ $s_{\scriptscriptstyle 4} = 1,$ 
$s_{\scriptscriptstyle 1} + s_{\scriptscriptstyle 4},\, s_{\scriptscriptstyle 2} + s_{\scriptscriptstyle 4},\, s_{\scriptscriptstyle 3} + s_{\scriptscriptstyle 4} = \frac{3}{2},$  $s_{\scriptscriptstyle 1} + s_{\scriptscriptstyle 2} + s_{\scriptscriptstyle 4}, \, s_{\scriptscriptstyle 1} + s_{\scriptscriptstyle 3} + s_{\scriptscriptstyle 4},\, s_{\scriptscriptstyle 2} + s_{\scriptscriptstyle 3} + s_{\scriptscriptstyle 4} = 2,$ 
$s_{\scriptscriptstyle 1} + s_{\scriptscriptstyle 2} + s_{\scriptscriptstyle 3} + s_{\scriptscriptstyle 4} = \frac{5}{2},$ 
$s_{\scriptscriptstyle 1} + s_{\scriptscriptstyle 2} + s_{\scriptscriptstyle 3} + 2s_{\scriptscriptstyle 4} = 3$ \cite{DGH}. 
We will be specializing at $s_{\scriptscriptstyle 1} = s_{\scriptscriptstyle 2} = s_{\scriptscriptstyle 3} = \frac{1}{2}$ and examining poles in $s_{\scriptscriptstyle 4}.$ 
The first three poles listed are irrelevant; the next eight specialize to poles at $s_{\scriptscriptstyle 4} = 1;$ the last one specializes to a pole at 
$s_{\scriptscriptstyle 4} = \frac{3}{4},$ which is our particular focus. We will compute the residue at $s_{\scriptscriptstyle 4} = 1$ directly, and then find the residue 
at $s_{\scriptscriptstyle 1} + s_{\scriptscriptstyle 2} + s_{\scriptscriptstyle 3} + 2s_{\scriptscriptstyle 4} = 3$ utilizing functional equations.

\vskip10pt
\begin{prop}
The function $Z^{(c)}\big(s_{\scriptscriptstyle 1}, s_{\scriptscriptstyle 2}, s_{\scriptscriptstyle 3}, s_{\scriptscriptstyle 4}; 
\chi_{a_{\scriptscriptstyle 2}c_{\scriptscriptstyle 2}}, \chi_{a_{\scriptscriptstyle 1}c_{\scriptscriptstyle 1}}\big)$ is holomorphic at $s_{\scriptscriptstyle 4} = 1$ if 
$\chi_{a_{\scriptscriptstyle 2}c_{\scriptscriptstyle 2}}$ is a nontrivial character. If $\chi_{a_{\scriptscriptstyle 2}c_{\scriptscriptstyle 2}}$ is trivial then it has a simple 
pole at $s_{\scriptscriptstyle 4} = 1$ with residue 
\begin{equation*}
\begin{split}
&R^{(c)}(s_{\scriptscriptstyle 1}, s_{\scriptscriptstyle 2}, s_{\scriptscriptstyle 3}) := 
\underset{s_{\scriptscriptstyle 4} = 1}{\Res} \; Z^{(c)}\big(s_{\scriptscriptstyle 1}, s_{\scriptscriptstyle 2}, s_{\scriptscriptstyle 3}, s_{\scriptscriptstyle 4}
; 1, \chi_{a_{\scriptscriptstyle 1}c_{\scriptscriptstyle 1}}\big) \\
& = \zeta^{(2c)}(2s_{\scriptscriptstyle 1})\zeta^{(2c)}(2s_{\scriptscriptstyle 2})\zeta^{(2c)}(2s_{\scriptscriptstyle 3})
\zeta^{(2c)}(s_{\scriptscriptstyle 1} + s_{\scriptscriptstyle 2})\zeta^{(2c)}(s_{\scriptscriptstyle 1} + s_{\scriptscriptstyle 3})
\zeta^{(2c)}(s_{\scriptscriptstyle 2} + s_{\scriptscriptstyle 3})\zeta^{(2c)}(2s_{\scriptscriptstyle 1} + 2s_{\scriptscriptstyle 2} + 
2s_{\scriptscriptstyle 3} - 1)\prod_{p\mid 2c}(1 - p^{-1}).
\end{split}
\end{equation*}
\end{prop} 

\begin{proof} We utilize the expression \eqref{MDS-vers2}. \!In this expression, the polynomials 
$
Q_{\scriptscriptstyle m_{\scriptscriptstyle 1},\, m_{\scriptscriptstyle 2},\, m_{\scriptscriptstyle 3}}
(s_{\scriptscriptstyle 4}; \chi_{a_{\scriptscriptstyle 2}c_{\scriptscriptstyle 2}}\tilde{\chi}_{m_{0}})
$ 
do not contribute poles at $s_{\scriptscriptstyle 4} = 1.$ Poles arise only when the quadratic $L$-function 
$
L^{(2c)}(s_{\scriptscriptstyle 4}, \chi_{a_{\scriptscriptstyle 2}c_{\scriptscriptstyle 2}}\tilde{\chi}_{m_{0}})
$ 
is actually a zeta function -- that is, when $a_{\scriptscriptstyle 2} = c_{\scriptscriptstyle 2} = 1$ and $m_{0} = 1,$ or, equivalently, 
$m_{\scriptscriptstyle 1}m_{\scriptscriptstyle 2}m_{\scriptscriptstyle 3}$ is a perfect square. In this case, $\zeta^{(2c)}(s_{\scriptscriptstyle 4})$ has a simple pole 
at $s_{\scriptscriptstyle 4} = 1,$\\ with residue $\prod_{p\mid 2c} (1 - p^{-1}).$ Thus 
$
Z^{(c)}\big(s_{\scriptscriptstyle 1}, s_{\scriptscriptstyle 2}, s_{\scriptscriptstyle 3}, s_{\scriptscriptstyle 4}; 
\chi_{a_{\scriptscriptstyle 2}c_{\scriptscriptstyle 2}}, \chi_{a_{\scriptscriptstyle 1}c_{\scriptscriptstyle 1}}\big)
$ 
is holomorphic at $s_{\scriptscriptstyle 4} = 1$ if $\chi_{a_{\scriptscriptstyle 2}c_{\scriptscriptstyle 2}}$ is a nontrivial character; if 
$\chi_{a_{\scriptscriptstyle 2}c_{\scriptscriptstyle 2}}$ is trivial then it has a simple pole at $s_{\scriptscriptstyle 4} = 1$ with residue 
\begin{equation*}
R^{(c)}(s_{\scriptscriptstyle 1}, s_{\scriptscriptstyle 2}, s_{\scriptscriptstyle 3}) = 
\prod_{p\mid 2c} (1 - p^{-1}) \cdot \sum_{\substack{m_{\scriptscriptstyle 1},\,  m_{\scriptscriptstyle 2}, \, m_{\scriptscriptstyle 3} \, \geq \, 1 
\\ \gcd(m_{\scriptscriptstyle 1}m_{\scriptscriptstyle 2}m_{\scriptscriptstyle 3}, 2c) = 1 \\ 
m_{\scriptscriptstyle 1}m_{\scriptscriptstyle 2}m_{\scriptscriptstyle 3}-\text{square}}} m_{\scriptscriptstyle 1}^{- s_{\scriptscriptstyle 1}}
m_{\scriptscriptstyle 2}^{- s_{\scriptscriptstyle 2}} m_{\scriptscriptstyle 3}^{- s_{\scriptscriptstyle 3}} 
Q_{\scriptscriptstyle m_{\scriptscriptstyle 1},\, m_{\scriptscriptstyle 2}, \, m_{\scriptscriptstyle 3}}(1; 1).
\end{equation*} 
This expression is independent of $a_{\scriptscriptstyle 1}$ and $c_{\scriptscriptstyle 1}.$ 

Recall that 
$
Q_{\scriptscriptstyle m_{\scriptscriptstyle 1}, \, m_{\scriptscriptstyle 2},\, m_{\scriptscriptstyle 3}}
(s_{\scriptscriptstyle 4}; \chi_{a_{\scriptscriptstyle 2}c_{\scriptscriptstyle 2}}\tilde{\chi}_{m_{0}})
$ 
was defined as a product over primes $p$ such that 
$p^{k_{\scriptscriptstyle 1}}\parallel m_{\scriptscriptstyle 1},$ $p^{k_{\scriptscriptstyle 2}}\parallel m_{\scriptscriptstyle 2},$ 
$p^{k_{\scriptscriptstyle 3}}\parallel m_{\scriptscriptstyle 3},$ and $k_{\scriptscriptstyle 1} + k_{\scriptscriptstyle 2} + k_{\scriptscriptstyle 3} \geq 2.$ 
It follows that the residue has an Euler product expression; the factor at $p$ for $p \nmid 2c$ is: 
\begin{equation*}
1 \; + \sum_{k_{\scriptscriptstyle 1} + k_{\scriptscriptstyle 2} + k_{\scriptscriptstyle 3}\, \geq \, 2\; \text{even}} 
Q_{\scriptscriptstyle p^{k_{\scriptscriptstyle 1}},\, p^{k_{\scriptscriptstyle 2}}, \, p^{k_{\scriptscriptstyle 3}}}(1; 1) 
p^{- k_{\scriptscriptstyle 1}s_{\scriptscriptstyle 1} - k_{\scriptscriptstyle 2}s_{\scriptscriptstyle 2} - k_{\scriptscriptstyle 3}s_{\scriptscriptstyle 3}}. 
\end{equation*}
This can be evaluated directly from the explicit generating function $Z_{\scriptscriptstyle p}(s_{\scriptscriptstyle 1}, s_{\scriptscriptstyle 2}, s_{\scriptscriptstyle 3}, s_{\scriptscriptstyle 4})$ as 
\begin{equation*}
\big(1 - p^{- 2s_{\scriptscriptstyle 1}}\big)^{-1} 
\big(1 - p^{- 2s_{\scriptscriptstyle 2}}\big)^{-1} 
\big(1 - p^{- 2s_{\scriptscriptstyle 3}}\big)^{-1} 
\big(1 - p^{- s_{\scriptscriptstyle 1} - s_{\scriptscriptstyle 2}}\big)^{-1} 
\big(1 - p^{- s_{\scriptscriptstyle 1} - s_{\scriptscriptstyle 3}}\big)^{-1} 
\big(1 - p^{- s_{\scriptscriptstyle 2} - s_{\scriptscriptstyle 3}}\big)^{-1} 
\big(1 - p^{1 - 2s_{\scriptscriptstyle 1} - 2s_{\scriptscriptstyle 2} - 2s_{\scriptscriptstyle 3}}\big)^{-1}
\end{equation*} 
from which the theorem follows.
\end{proof} 
To simplify the computation of the second residue, we restrict to the situation of particular interest to us: fix 
$
s_{\scriptscriptstyle 1} = s_{\scriptscriptstyle 2} = s_{\scriptscriptstyle 3} = \frac{1}{2},
$ 
$s_{\scriptscriptstyle 4} = s.$ We also assume that $\gcd(c_{\scriptscriptstyle 1}, c_{\scriptscriptstyle 2}) = 1,$ $a_{\scriptscriptstyle 1} = 2,$ and 
$a_{\scriptscriptstyle 2} = (-1)^{(c_{\scriptscriptstyle 2} - 1)\slash 2}$ (so that 
$\chi_{a_{\scriptscriptstyle 2}c_{\scriptscriptstyle 2}} = \tilde{\chi}_{c_{\scriptscriptstyle 2}}$). 

\vskip10pt
\begin{prop} \label{residue}
Suppose that $c = c_{\scriptscriptstyle 1}c_{\scriptscriptstyle 2}c_{\scriptscriptstyle 3}.$ Then 
$
Z^{(c)}\big(\frac{1}{2}, \frac{1}{2}, \frac{1}{2}, s; \tilde{\chi}_{c_{\scriptscriptstyle 2}}, \chi_{2c_{\scriptscriptstyle 1}}\big)
$ 
has a simple pole at $s = \frac{3}{4},$ with residue 
\begin{equation} \label{mds-residue}
\begin{split}
\underset{s = \frac{3}{4}}{\Res} \; Z^{(c)}\big(\tfrac{1}{2}, \tfrac{1}{2}, \tfrac{1}{2}, s; \tilde{\chi}_{c_{\scriptscriptstyle 2}}, \chi_{2c_{\scriptscriptstyle 1}}\big) \,
= \, &\tfrac{9}{256\pi} 2^{\frac{1}{4}} (- 181 + 128\sqrt{2}\,) \Gamma\left(\tfrac{1}{4}\right)^{\! 4} 
\zeta\left(\tfrac{1}{2}\right)^{\! 7} \!\left(\tfrac{2c_{\scriptscriptstyle 1}}{c_{\scriptscriptstyle 2}}\right) \\
&\cdot c_{\scriptscriptstyle 1}^{- \frac{1}{4}}\prod_{p \mid c_{\scriptscriptstyle 1}}\big(1 - p^{- \frac{1}{2}}\big)^{\! 8} 
\big(1 + p^{- \frac{1}{2}}\big)^{\! 2} \big(1 + 6p^{- \frac{1}{2}} + p^{-1}\big) \\
&\cdot c_{\scriptscriptstyle 2}^{- \frac{1}{2}}\prod_{p\mid c_{\scriptscriptstyle 2}} 
\big(1 - p^{- \frac{1}{2}} \big)^{\! 8} \big(1 + p^{- \frac{1}{2}}\big) \big(3 + 7p^{- \frac{1}{2}} + 3p^{-1} \big) \\ 
&\cdot \prod_{p\mid c_{\scriptscriptstyle 3}} \big(1 - p^{- \frac{1}{2}}\big)^{\! 8} 
\big(1 + p^{- \frac{1}{2}}\big) \big(1 + 7p^{- \frac{1}{2}} + 13p^{-1} + 7p^{- \frac{3}{2}} + p^{-2}\big).
\end{split}
\end{equation}
\end{prop}

\begin{proof}
Apply the functional equations $\sigma_{\scriptscriptstyle 1} \sigma_{\scriptscriptstyle 2} \sigma_3 \sigma_4$:
\begin{align*}
&Z^{(c)}\big(s_{\scriptscriptstyle 1}, s_{\scriptscriptstyle 2}, s_{\scriptscriptstyle 3}, s_{\scriptscriptstyle 4}; \chi_{a_{\scriptscriptstyle 2}c_{\scriptscriptstyle 2}}, \chi_{a_{\scriptscriptstyle 1}c_{\scriptscriptstyle 1}}\big) \\
&=2^{-2\omega(c)-4} \sum_{D, \, M \in \frac{(\ZZ/8c\ZZ)^*}{(\ZZ/8c\ZZ)^{*2}}} \chi_{a_{\scriptscriptstyle 1}c_{\scriptscriptstyle 1}}(M) \chi_{a_{\scriptscriptstyle 2}c_{\scriptscriptstyle 2}}(D) \sum_{\substack{a_{\scriptscriptstyle 1}', \, a_{\scriptscriptstyle 2}'\in \{\pm 1, \pm 2\}\\ c_{\scriptscriptstyle 1}', \, c_{\scriptscriptstyle 2}'
\mid c}} \chi_{a_{\scriptscriptstyle 1}'c_{\scriptscriptstyle 1}'}(M)\chi_{a_{\scriptscriptstyle 2}'c_{\scriptscriptstyle 2}'}(D) \nonumber \\
&\cdot \frac{\Lambda_{c}\big(\tfrac{3}{2}-s_{\scriptscriptstyle 1}-s_{\scriptscriptstyle 4}; \chi_{a_{\scriptscriptstyle 1}'c_{\scriptscriptstyle 1}'}\chi_{D} \big)
\Lambda_{c}\big(\tfrac{3}{2}-s_{\scriptscriptstyle 2}-s_{\scriptscriptstyle 4}; \chi_{a_{\scriptscriptstyle 1}'c_{\scriptscriptstyle 1}'}\chi_{D} \big)
\Lambda_{c}\big(\tfrac{3}{2}-s_{\scriptscriptstyle 3}-s_{\scriptscriptstyle 4}; \chi_{a_{\scriptscriptstyle 1}'c_{\scriptscriptstyle 1}'}\chi_{D} \big)
\Lambda_{c}(1-s_{\scriptscriptstyle 4}; \chi_{a_{\scriptscriptstyle 2}c_{\scriptscriptstyle 2}}\tilde{\chi}_M)}
{\Lambda_{c}\big(s_{\scriptscriptstyle 1}+s_{\scriptscriptstyle 4}-\frac{1}{2}; \chi_{a_{\scriptscriptstyle 1}'c_{\scriptscriptstyle 1}'}\chi_{D} \big)
\Lambda_{c}\big(s_{\scriptscriptstyle 2}+s_{\scriptscriptstyle 4}-\frac{1}{2}; \chi_{a_{\scriptscriptstyle 1}'c_{\scriptscriptstyle 1}'}\chi_{D} \big) 
\Lambda_{c}\big(s_{\scriptscriptstyle 3}+s_{\scriptscriptstyle 4}-\tfrac{1}{2}; \chi_{a_{\scriptscriptstyle 1}'c_{\scriptscriptstyle 1}'}\chi_{D} \big)
\Lambda_{c}(s_{\scriptscriptstyle 4}; \chi_{a_{\scriptscriptstyle 2}c_{\scriptscriptstyle 2}}\tilde{\chi}_M)}  \nonumber \\
&\cdot Z^{(c)}\big(\tfrac{3}{2}-s_{\scriptscriptstyle 1}-s_{\scriptscriptstyle 4}, \tfrac{3}{2}-s_{\scriptscriptstyle 2}-s_{\scriptscriptstyle 4}, \tfrac{3}{2} 
- s_{\scriptscriptstyle 3}-s_{\scriptscriptstyle 4}, s_{\scriptscriptstyle 1}+s_{\scriptscriptstyle 2}+s_{\scriptscriptstyle 3}+2s_{\scriptscriptstyle 4}-2; \chi_{a_{\scriptscriptstyle 2}' c_{\scriptscriptstyle 2}'}, \chi_{a_{\scriptscriptstyle 1}' c_{\scriptscriptstyle 1}'}\big).
\end{align*}
The simple pole at $s_{\scriptscriptstyle 1}+s_{\scriptscriptstyle 2}+s_{\scriptscriptstyle 3}+2s_{\scriptscriptstyle 4}=3$ arises from summands of 
\begin{equation*}
Z^{(c)}\big(
\tfrac{3}{2}-s_{\scriptscriptstyle 1} - s_{\scriptscriptstyle 4}, 
\tfrac{3}{2}-s_{\scriptscriptstyle 2} - s_{\scriptscriptstyle 4}, 
\tfrac{3}{2}-s_{\scriptscriptstyle 3} - s_{\scriptscriptstyle 4}, 
2s_{\scriptscriptstyle 4} + s_{\scriptscriptstyle 1}+s_{\scriptscriptstyle 2}+s_{\scriptscriptstyle 3}-2; 
\chi_{a_{\scriptscriptstyle 2}' c_{\scriptscriptstyle 2}'}, \chi_{a_{\scriptscriptstyle 1}' c_{\scriptscriptstyle 1}'}\big)
\end{equation*} 
in this expression with $a_{\scriptscriptstyle 2}' = c_{\scriptscriptstyle 2}'=1.$ The full residue is:
\begin{align*}
&\underset{s_{\scriptscriptstyle 1}+s_{\scriptscriptstyle 2}+s_{\scriptscriptstyle 3}+2s_{\scriptscriptstyle 4}=3}{\Res} \; Z^{(c)}\big(s_{\scriptscriptstyle 1}, s_{\scriptscriptstyle 2}, s_{\scriptscriptstyle 3}, s_{\scriptscriptstyle 4}; \chi_{a_{\scriptscriptstyle 2}c_{\scriptscriptstyle 2}}, \chi_{a_{\scriptscriptstyle 1}c_{\scriptscriptstyle 1}}\big) \\
&=R^{(c)}\big(\tfrac{3}{2}-s_{\scriptscriptstyle 1}-s_{\scriptscriptstyle 4}, \tfrac{3}{2}-s_{\scriptscriptstyle 2}-s_{\scriptscriptstyle 4}, 
\tfrac{3}{2}-s_{\scriptscriptstyle 3}-s_{\scriptscriptstyle 4}\big)\, 2^{-2\omega(c)-4} 
\sum_{D, \, M \in \frac{(\ZZ/8c\ZZ)^*}{(\ZZ/8c\ZZ)^{*2}}} \chi_{a_{\scriptscriptstyle 1}c_{\scriptscriptstyle 1}}(M) 
\chi_{a_{\scriptscriptstyle 2}c_{\scriptscriptstyle 2}}(D) 
\sum_{\substack{a_{\scriptscriptstyle 1}' \in \{\pm 1, \pm 2\}\\ c_{\scriptscriptstyle 1}' \mid c}} \chi_{a_{\scriptscriptstyle 1}' c_{\scriptscriptstyle 1}'}(M) \\
&\cdot 
\frac{\Lambda_{c}\big(\tfrac{3}{2}-s_{\scriptscriptstyle 1}-s_{\scriptscriptstyle 4}; \chi_{a_{\scriptscriptstyle 1}'c_{\scriptscriptstyle 1}'}\chi_{D}\big) 
\Lambda_{c}\big(\tfrac{3}{2}-s_{\scriptscriptstyle 2}-s_{\scriptscriptstyle 4}; \chi_{a_{\scriptscriptstyle 1}'c_{\scriptscriptstyle 1}'}\chi_{D} \big) 
\Lambda_{c}\big(\tfrac{3}{2}-s_{\scriptscriptstyle 3}-s_{\scriptscriptstyle 4}; \chi_{a_{\scriptscriptstyle 1}'c_{\scriptscriptstyle 1}'}\chi_{D} \big)
\Lambda_{c}(1 - s_{\scriptscriptstyle 4}; \chi_{a_{\scriptscriptstyle 2}c_{\scriptscriptstyle 2}}\tilde{\chi}_M)}
{\Lambda_{c}\big(s_{\scriptscriptstyle 1}+s_{\scriptscriptstyle 4}-\tfrac{1}{2}; \chi_{a_{\scriptscriptstyle 1}'c_{\scriptscriptstyle 1}'}\chi_{D} \big) 
\Lambda_{c}\big(s_{\scriptscriptstyle 2}+s_{\scriptscriptstyle 4}-\tfrac{1}{2}; \chi_{a_{\scriptscriptstyle 1}'c_{\scriptscriptstyle 1}'}\chi_{D}\big) 
\Lambda_{c}\big(s_{\scriptscriptstyle 3}+s_{\scriptscriptstyle 4}-\tfrac{1}{2}; \chi_{a_{\scriptscriptstyle 1}'c_{\scriptscriptstyle 1}'}\chi_{D}\big) 
\Lambda_{c}(s_{\scriptscriptstyle 4}; \chi_{a_{\scriptscriptstyle 2}c_{\scriptscriptstyle 2}}\tilde{\chi}_M)}.
\end{align*}
To proceed, we adopt the hypotheses of the proposition. We must also divide by a factor of $2$ to translate the residue at $2s=\frac{3}{2}$ to the residue at 
$s=\frac{3}{4}.$ The result is  
\begin{align*}
&\underset{s = \frac{3}{4}}{\Res} \; Z^{(c)}\big(\tfrac{1}{2}, \tfrac{1}{2}, \tfrac{1}{2}, s; \tilde{\chi}_{c_{\scriptscriptstyle 2}}, \chi_{2c_{\scriptscriptstyle 1}}\big) \\
& = R^{(c)}\big(\tfrac{1}{4}, \tfrac{1}{4}, \tfrac{1}{4}\big) \, 2^{- 2\omega(c) - 5} 
\sum_{D, \, M \in \frac{(\ZZ/8c\ZZ)^*}{(\ZZ/8c\ZZ)^{*2}}} \chi_{2c_{\scriptscriptstyle 1}}(M) \tilde{\chi}_{c_{\scriptscriptstyle 2}}(D) 
\sum_{\substack{a_{\scriptscriptstyle 1}' \in \{\pm 1, \pm 2\}\\ c_{\scriptscriptstyle 1}' \mid c}} 
\chi_{a_{\scriptscriptstyle 1}' c_{\scriptscriptstyle 1}'}(M) \,
\frac{\Lambda_{c}\big(\tfrac{1}{4}; \chi_{a_{\scriptscriptstyle 1}'c_{\scriptscriptstyle 1}'}\chi_{D}\big)^{\! 3} 
\Lambda_{c}\big(\tfrac{1}{4}; \tilde{\chi}_{c_{\scriptscriptstyle 2}}\tilde{\chi}_{M}\big)}
{\Lambda_{c}\big(\tfrac{3}{4}; \chi_{a_{\scriptscriptstyle 1}'c_{\scriptscriptstyle 1}'}\chi_{D} \big)^{\! 3} 
\Lambda_{c}\big(\tfrac{3}{4}; \tilde{\chi}_{c_{\scriptscriptstyle 2}}\tilde{\chi}_{M}\big)}.
\end{align*} 
We may apply the following evaluations:
\begin{align*}
&\frac{\Lambda_{c}\big(\tfrac{1}{4}; \chi_{a_{\scriptscriptstyle 1}'c_{\scriptscriptstyle 1}'}\chi_{D}\big)^{\! 3}}
{\Lambda_{c}\big(\tfrac{3}{4}; \chi_{a_{\scriptscriptstyle 1}'c_{\scriptscriptstyle 1}'}\chi_{D} \big)^{\! 3}} = 
\Big(\tfrac{\pi}{\gcd(\cond \chi_{a_{\scriptscriptstyle 1}'c_{\scriptscriptstyle 1}'}\chi_D, \, 8c)}\Big)^{\! \frac{3}{4}} 
\prod_{\substack{p \, \mid \, 2c \\ p \, \nmid \, \cond \chi_{a_{\scriptscriptstyle 1}'c_{\scriptscriptstyle 1}'}\chi_D}} 
\big(1+\chi_{a_{\scriptscriptstyle 1}'c_{\scriptscriptstyle 1}'}(p)\chi_D(p) p^{-\frac{1}{4}} +p^{-\frac{1}{2}} \big)^{\! 3} 
\cdot \left\lbrace \begin{array}{cc} 
\frac{\Gamma(1/8)^3}{\Gamma(3/8)^3} & \text{if $a_{\scriptscriptstyle 1}'  > 0$} \\ 
\frac{\Gamma(5/8)^3}{\Gamma(7/8)^3}  & \text{if $a_{\scriptscriptstyle 1}' < 0$} \end{array} \right. \\
&\frac{\Lambda_{c}\big(\tfrac{1}{4}; \tilde{\chi}_{c_{\scriptscriptstyle 2}}\tilde{\chi}_{M} \big)}
{\Lambda_{c}\big(\tfrac{3}{4}; \tilde{\chi}_{c_{\scriptscriptstyle 2}}\tilde{\chi}_{M} \big)} 
= \big(\tfrac{\pi}{c_{\scriptscriptstyle 2}}\big)^{\! \frac{1}{4}} 
\! \prod_{p \, \mid \, (2c \slash c_{\scriptscriptstyle 2})} \big(1+\tilde{\chi}_{c_{\scriptscriptstyle 2}}(p)\tilde{\chi}_M(p) p^{-\frac{1}{4}}+p^{-\frac{1}{2}} \big) 
\cdot \left\lbrace \begin{array}{cc} \frac{\Gamma(1/8)}{\Gamma(3/8)} & \text{if\, $c_{\scriptscriptstyle 2}\, M \equiv 1 \!\!\!\!\!\pmod 4$} \\ 
\frac{\Gamma(5/8)}{\Gamma(7/8)}  & \text{if $c_{\scriptscriptstyle 2}\, M \equiv 3 \!\!\!\!\!\pmod 4$.} \end{array} \right.
\end{align*} 
After expansion of 
$
\chi_{a_{\scriptscriptstyle 2}c_{\scriptscriptstyle 2}}(D) \frac{\Lambda_{c}\big(\frac{1}{4};\, \chi_{a_{\scriptscriptstyle 1}' c_{\scriptscriptstyle 1}'}\chi_{D}\big)^{\! 3}}{\Lambda_{c}\big(\frac{3}{4};\,  \chi_{a_{\scriptscriptstyle 1}' c_{\scriptscriptstyle 1}'}\chi_{D} \big)^{\! 3}}
$ 
and summation over $D,$ each term will vanish, unless its total character of \\ 
$D$ is trivial. This is only possible for $c_{\scriptscriptstyle 1}'$ coprime to $c_{\scriptscriptstyle 2}.$ Similarly, after expansion of 
$
\chi_{2c_{\scriptscriptstyle 1}}(M)\chi_{a_{\scriptscriptstyle 1}' c_{\scriptscriptstyle 1}'}(M)
\frac{\Lambda_{c}\big(\frac{1}{4};\, \tilde{\chi}_{c_{\scriptscriptstyle 2}}\tilde{\chi}_{M} \big)}
{\Lambda_{c}\big(\frac{3}{4};\, \tilde{\chi}_{c_{\scriptscriptstyle 2}}\tilde{\chi}_{M} \big)} 
$ 
and \\ summation in $M,$ each term will vanish, unless its total character of $M$ is trivial. 
\!The residue after both these summations is as follows: 
\begin{align*}
& R^{(c)}\big(\tfrac{1}{4}, \tfrac{1}{4}, \tfrac{1}{4}\big) \big(\tfrac{8c_{\scriptscriptstyle 1}}{c_{\scriptscriptstyle 2}}\big) 
\cdot \tfrac{1}{2}\sum_{\substack{a_{\scriptscriptstyle 1}' \in \{\pm 1, \pm 2\}\\ c_{\scriptscriptstyle 1}' \, \mid \, (c\slash c_{\scriptscriptstyle 2})}} 
c_{\scriptscriptstyle 1}'^{-\frac{3}{4}} 
\Big(\tfrac{c_{\scriptscriptstyle 1}^{}c_{\scriptscriptstyle 1}'}{\gcd(c_{\scriptscriptstyle 1}^{},\, c_{\scriptscriptstyle 1}')^{\scriptscriptstyle 2}}\Big)^{\!\! -\frac{1}{4}}  c_{\scriptscriptstyle 2}^{-\frac{1}{2}} \prod_{p\, \mid \, c_{\scriptscriptstyle 2}}\big(3+7p^{-\frac{1}{2}}+3p^{-1}\big) \\
&\cdot \prod_{p\, \mid \, (c\gcd(c_{\scriptscriptstyle 1}^{},\, c_{\scriptscriptstyle 1}')^{\scriptscriptstyle 2}\slash 
c_{\scriptscriptstyle 1}^{}c_{\scriptscriptstyle 1}'c_{\scriptscriptstyle 2}^{})} \big(1+p^{-\frac{1}{2}}\big)  
\prod_{p \, \mid \, (c\slash c_{\scriptscriptstyle 1}'c_{\scriptscriptstyle 2}^{})} \big(1+6p^{-\frac{1}{2}}+6p^{-1}+p^{-\frac{3}{2}}\big) \\
& \cdot \pi \left(\tfrac{\Gamma(1/8)}{\Gamma(3/8)} \, + \, \chi_{a_{\scriptscriptstyle 1}'}(-1)\tfrac{\Gamma(5/8)}{\Gamma(7/8)}\right) 
\cdot\left\lbrace \begin{array}{cc} \frac{\Gamma(1/8)^3}{\Gamma(3/8)^3} & \text{if $\chi_{a_{\scriptscriptstyle 1}'}(-1)=1$} \\ 
\frac{\Gamma(5/8)^3}{\Gamma(7/8)^3}  &  \;\;\;\, \text{if $\chi_{a_{\scriptscriptstyle 1}'}(-1)=-1$} \end{array} \right. 
\cdot \, \left\lbrace \begin{array}{cc} 2^{-\frac{1}{4}} \, +\, 7\cdot 2^{-\frac{11}{4}} & \text{if $2 \nmid a_{\scriptscriptstyle 1}'$} \\ 
2^{-\frac{13}{4}} \, + \, 2^{-\frac{15}{4}} &\; \text{if $2 \mid  a_{\scriptscriptstyle 1}'$.} \end{array} \right. 
\end{align*} 
The sum over $c_{\scriptscriptstyle 1}'$ in the first two lines can be expressed as a product over primes $p$ dividing $c.$ For each prime there are two cases depending on whether it divides $c_{\scriptscriptstyle 1}'.$ An Euler factor in this sum has the form
\begin{equation*}
\left\lbrace\begin{array}{cc} 
p^{-\frac{1}{4}}\big(1+p^{-\frac{1}{2}}\big)\big(1+6p^{-\frac{1}{2}}+p^{-1}\big) & \text{if $p\mid c_{\scriptscriptstyle 1}$} \\
p^{-\frac{1}{2}} \big(3+7p^{-\frac{1}{2}}+3p^{-1}\big) & \text{if $p\mid c_{\scriptscriptstyle 2}$} \\
\big(1+7p^{-\frac{1}{2}}+13p^{-1}+7p^{-\frac{3}{2}}+p^{-2}\big) &\;\, \text{if $p\mid c_{\scriptscriptstyle 3}$}.
\end{array}\right.
\end{equation*} 
The sum of the final line over the four possible values of $a_{\scriptscriptstyle 1}'$ yields a constant representing the contribution of the prime $2$ 
and the archimedean place. \!After simplification this constant is $\frac{9}{8}2^{\frac{1}{4}}(6+5\sqrt{2})\Gamma\left(\frac{1}{4}\right)^{\! 4}\slash \pi.$ 
Combining these two computations with the result of the previous proposition yields the desired formula. 
\end{proof}

We remark that the selection of $a_{\scriptscriptstyle 1} = 2,$ $a_{\scriptscriptstyle 2} = (-1)^{(c_{\scriptscriptstyle 2} - 1)\slash 2}$ is made 
in order to isolate fundamental discriminants $d$ which are positive and divisible by $8.$ Analogous computations could be made with other 
choices of $a_{\scriptscriptstyle 1},$ $a_{\scriptscriptstyle 2}$ to isolate other types of fundamental discriminants.

\section{The Sieve} \label{Sieve} 
{\bf Construction of the sieve.} For any square-free odd positive integer $h$ and $a_{\scriptscriptstyle 1}, a_{\scriptscriptstyle 2} 
\in \{\pm 1, \pm 2\},$ define 
\begin{equation*}
Z(s_{\scriptscriptstyle 1}, s_{\scriptscriptstyle 2}, s_{\scriptscriptstyle 3}, s_{\scriptscriptstyle 4}, 
\chi_{a_{\scriptscriptstyle 2}}, \chi_{a_{\scriptscriptstyle 1}}\!; \, h)\;\;  = 
\sum_{\substack{m_{\scriptscriptstyle 1}\!, \, m_{\scriptscriptstyle 2}, \, m_{\scriptscriptstyle 3}, 
\, d \, \geq \, 1\; \text{odd} \\ h^{2}\mid d}}  \, H(m_{\scriptscriptstyle 1}, m_{\scriptscriptstyle 2}, m_{\scriptscriptstyle 3}, d)
\chi_{a_{\scriptscriptstyle 1}}\!(m_{\scriptscriptstyle 1}m_{\scriptscriptstyle 2}m_{\scriptscriptstyle 3})
\chi_{a_{\scriptscriptstyle 2}}\!(d)
{m_{\scriptscriptstyle 1}^{- s_{\scriptscriptstyle 1}}
m_{\scriptscriptstyle 2}^{- s_{\scriptscriptstyle 2}} 
m_{\scriptscriptstyle 3}^{- s_{\scriptscriptstyle 3}} 
d_{}^{- s_{\scriptscriptstyle 4}}} 
\end{equation*}
and 
\begin{equation*}
Z_{\scriptscriptstyle 0}(s_{\scriptscriptstyle 1},s_{\scriptscriptstyle 2},s_{\scriptscriptstyle 3},s_{\scriptscriptstyle 4}, \chi_{a_{\scriptscriptstyle 2}}, \chi_{a_{\scriptscriptstyle 1}}) \;\; = 
\sum_{\substack{d_{\scriptscriptstyle 0} \, >  \, 0 \\ 
d_{\scriptscriptstyle 0} - \mathrm{odd \; \& \; sq. \; free}}}  
L^{\scriptscriptstyle (2)}(s_{\scriptscriptstyle 1}, \chi_{a_{\scriptscriptstyle 1} d_{\scriptscriptstyle 0}}) \, 
L^{\scriptscriptstyle (2)}(s_{\scriptscriptstyle 2}, \chi_{a_{\scriptscriptstyle 1} d_{\scriptscriptstyle 0}}) \,  
L^{\scriptscriptstyle (2)}(s_{\scriptscriptstyle 3}, \chi_{a_{\scriptscriptstyle 1} d_{\scriptscriptstyle 0}}) \, 
\chi_{a_{\scriptscriptstyle 2}}\!(d_{\scriptscriptstyle 0})\,  
d_{\scriptscriptstyle 0}^{ - s_{\scriptscriptstyle 4}}.
\end{equation*} 
As in \cite[Lemma~5.1]{D}, we can write 
\begin{equation} \label{mobius}
Z_{\scriptscriptstyle 0}(s_{\scriptscriptstyle 1},s_{\scriptscriptstyle 2},s_{\scriptscriptstyle 3},s_{\scriptscriptstyle 4}, \chi_{a_{\scriptscriptstyle 2}}, \chi_{a_{\scriptscriptstyle 1}})\;\, =\sum_{(h, 2) = 1}\,  \mu(h) Z(s_{\scriptscriptstyle 1},s_{\scriptscriptstyle 2},s_{\scriptscriptstyle 3},s_{\scriptscriptstyle 4}, \chi_{a_{\scriptscriptstyle 2}}, \chi_{a_{\scriptscriptstyle 1}}\!; \, h).
\end{equation} 
The function $Z(s_{\scriptscriptstyle 1},s_{\scriptscriptstyle 2},s_{\scriptscriptstyle 3},s_{\scriptscriptstyle 4}, \chi_{a_{\scriptscriptstyle 2}}, \chi_{a_{\scriptscriptstyle 1}}\!; \, h)$ 
can, \!in turn, \!be expressed in terms of the multiple {D}irichlet series we have discussed in the previous sections. \!To state the relation of these functions, let us first define
\begin{equation*} 
F\!(z_{\scriptscriptstyle 1}^{}\!, \, z_{\scriptscriptstyle 2}^{}, \, z_{\scriptscriptstyle 3}^{}, \, z_{\scriptscriptstyle 4}^{}; p) 
\;\; : =  \sum_{k_{\scriptscriptstyle 1}\!, \, k_{\scriptscriptstyle 2}, \, k_{\scriptscriptstyle 3}, \, k \, \ge \, 0} 
\, H(p^{k_{\scriptscriptstyle 1}}\!, p^{k_{\scriptscriptstyle 2}}\!, p^{k_{\scriptscriptstyle 3}}\!, p^{2k+3})
\, z_{\scriptscriptstyle 1}^{\scriptscriptstyle k_{\scriptscriptstyle 1}} 
z_{\scriptscriptstyle 2}^{\scriptscriptstyle k_{\scriptscriptstyle 2}}
z_{\scriptscriptstyle 3}^{\scriptscriptstyle k_{\scriptscriptstyle 3}} 
z_{\scriptscriptstyle 4}^{\scriptscriptstyle 2k} 
\end{equation*} 
and, for $a\in \{0, 1\},$ 
\begin{equation*} 
G^{\scriptscriptstyle (a)}(z_{\scriptscriptstyle 1}^{}\!, \, z_{\scriptscriptstyle 2}^{}, \, z_{\scriptscriptstyle 3}^{}, \, z_{\scriptscriptstyle 4}^{}; p)  \;\; :  =  
\sum_{\substack{k_{\scriptscriptstyle 1}\!, \, k_{\scriptscriptstyle 2}, \, k_{\scriptscriptstyle 3}, \, k \, \ge \, 0\\ 
k_{\scriptscriptstyle 1} + \, k_{\scriptscriptstyle 2} + \, k_{\scriptscriptstyle 3} \, \equiv \, a \!\!\!\!\!\pmod 2}} 
\,  H(p^{k_{\scriptscriptstyle 1}}\!, p^{k_{\scriptscriptstyle 2}}\!, p^{k_{\scriptscriptstyle 3}}\!, p^{2k+2})
\, z_{\scriptscriptstyle 1}^{\scriptscriptstyle k_{\scriptscriptstyle 1}} 
z_{\scriptscriptstyle 2}^{\scriptscriptstyle k_{\scriptscriptstyle 2}}
z_{\scriptscriptstyle 3}^{\scriptscriptstyle k_{\scriptscriptstyle 3}} 
z_{\scriptscriptstyle 4}^{\scriptscriptstyle 2k}.
\end{equation*} 
These are rational functions which can be computed explicitly from the $p$-part 
$
Z_{\scriptscriptstyle p}(s_{\scriptscriptstyle 1}, s_{\scriptscriptstyle 2}, s_{\scriptscriptstyle 3}, s_{\scriptscriptstyle 4}).$ 
Then, \!as in \cite[Section~5]{D}, one shows that\footnote{The variables $c$, $c'$, and $c_{\varepsilon}'$ in \cite[Section~5]{D} correspond to $c_{\scriptscriptstyle 1}$, $c_{\scriptscriptstyle 2}c_{\scriptscriptstyle 3}$, and $c_{\scriptscriptstyle 2}$ here.} 
\begin{equation} \label{eq: fundamental-eq-h}
\begin{split}
Z(s_{\scriptscriptstyle 1}, s_{\scriptscriptstyle 2}, s_{\scriptscriptstyle 3}, s_{\scriptscriptstyle 4}, 
& \chi_{a_{\scriptscriptstyle 2}}, \chi_{a_{\scriptscriptstyle 1}}\!; \, h) \\
= \, h^{- 2 s_{\scriptscriptstyle 4}} \!\!\sum_{h = c_{\scriptscriptstyle 1}c_{\scriptscriptstyle 2}c_{\scriptscriptstyle 3}} \;
\, & \chi_{a_{\scriptscriptstyle 2}}\!(c_{\scriptscriptstyle 1}) \chi_{a_{\scriptscriptstyle 1}}\!(c_{\scriptscriptstyle 2}) 
\Big(\tfrac{c_{\scriptscriptstyle 1}}{c_{\scriptscriptstyle 2}} \Big) 
Z^{(h)}(s_{\scriptscriptstyle 1}, s_{\scriptscriptstyle 2}, s_{\scriptscriptstyle 3}, s_{\scriptscriptstyle 4}; 
\chi_{a_{\scriptscriptstyle 2}}\tilde{\chi}_{c_{\scriptscriptstyle 2}}\!, 
\chi_{a_{\scriptscriptstyle 1}c_{\scriptscriptstyle 1}}) 
\prod_{p \, \mid \, c_{\scriptscriptstyle 1}} 
F(p^{\, - s_{\scriptscriptstyle 1}}\!, p^{\, - s_{\scriptscriptstyle 2}}\!, p^{\, - s_{\scriptscriptstyle 3}}\!, p^{\, - s_{\scriptscriptstyle 4}}; p) p^{\, - s_{\scriptscriptstyle 4}} \\
& \cdot\prod_{p \, \mid \, c_{\scriptscriptstyle 2}} 
G^{(1)}(p^{\, - s_{\scriptscriptstyle 1}}\!, p^{\, - s_{\scriptscriptstyle 2}}\!, p^{\, - s_{\scriptscriptstyle 3}}\!, p^{\, - s_{\scriptscriptstyle 4}}; p) 
\prod_{p \, \mid \, c_{\scriptscriptstyle 3}} G^{(0)}(p^{\, - s_{\scriptscriptstyle 1}}\!, p^{\, - s_{\scriptscriptstyle 2}}\!, 
p^{\, - s_{\scriptscriptstyle 3}}\!, p^{\, - s_{\scriptscriptstyle 4}}; p). 
\end{split}
\end{equation}
Note that the right-hand side of \eqref{eq: fundamental-eq-h} yields the meromorphic continuation of 
$
Z(s_{\scriptscriptstyle 1},s_{\scriptscriptstyle 2},s_{\scriptscriptstyle 3},s_{\scriptscriptstyle 4}, \chi_{a_{\scriptscriptstyle 2}}, \chi_{a_{\scriptscriptstyle 1}}\!; \, h).
$

\vskip10pt 
{\bf Recursive refinement of estimates.} For complex 
$
z_{\scriptscriptstyle 1}^{}\!, \, z_{\scriptscriptstyle 2}^{}, \, z_{\scriptscriptstyle 3}^{}, \, z_{\scriptscriptstyle 4}^{} 
$ 
and prime $p \ge 3,$ let 
\begin{align*} 
&f_{\scriptscriptstyle \mathrm{odd}}(z_{\scriptscriptstyle 1}^{}\!, \, z_{\scriptscriptstyle 2}^{}, \, z_{\scriptscriptstyle 3}^{}, 
\, z_{\scriptscriptstyle 4}^{}; p) \;\; : =  \sum_{k_{\scriptscriptstyle 1}\!, \, k_{\scriptscriptstyle 2}, \, k_{\scriptscriptstyle 3}, 
\, k \, \ge \, 0} \, H(p^{k_{\scriptscriptstyle 1}}\!, p^{k_{\scriptscriptstyle 2}}\!, p^{k_{\scriptscriptstyle 3}}\!, p^{2k+1}) 
\, z_{\scriptscriptstyle 1}^{\scriptscriptstyle k_{\scriptscriptstyle 1}} 
z_{\scriptscriptstyle 2}^{\scriptscriptstyle k_{\scriptscriptstyle 2}}
z_{\scriptscriptstyle 3}^{\scriptscriptstyle k_{\scriptscriptstyle 3}} 
z_{\scriptscriptstyle 4}^{\scriptscriptstyle 2k + 1} \\
&f_{\scriptscriptstyle \mathrm{even}}^{\, \pm}(z_{\scriptscriptstyle 1}^{}\!, \, z_{\scriptscriptstyle 2}^{}, 
\, z_{\scriptscriptstyle 3}^{}, \, z_{\scriptscriptstyle 4}^{}; p)  \;\; :  =  
\sum_{\substack{k_{\scriptscriptstyle 1}\!, \, k_{\scriptscriptstyle 2}, \, k_{\scriptscriptstyle 3}, \, k \, \ge \, 0\\ 
{\scriptscriptstyle (-1)^{k_{\scriptscriptstyle 1} + \, k_{\scriptscriptstyle 2} + \, k_{\scriptscriptstyle 3}}} \, = \, \pm 1}} 
\, H(p^{k_{\scriptscriptstyle 1}}\!, p^{k_{\scriptscriptstyle 2}}\!, p^{k_{\scriptscriptstyle 3}}\!, p^{2k}) 
\, z_{\scriptscriptstyle 1}^{\scriptscriptstyle k_{\scriptscriptstyle 1}} 
z_{\scriptscriptstyle 2}^{\scriptscriptstyle k_{\scriptscriptstyle 2}}
z_{\scriptscriptstyle 3}^{\scriptscriptstyle k_{\scriptscriptstyle 3}} 
z_{\scriptscriptstyle 4}^{\scriptscriptstyle 2k}.
\end{align*} 
We begin with the following lemma, see also \cite[Lemma~6.2]{D}.

\vskip10pt
\begin{lem}\label{Estimate inverse even part zeven} --- For every prime $p > 2$ and 
$|z| \le p^{\scriptscriptstyle - \frac{1}{2}}\!,$ we have 
\begin{equation*} 
\big| f_{\scriptscriptstyle \mathrm{odd}}\big(p^{\scriptscriptstyle -\frac{1}{2}}\!, \, p^{\scriptscriptstyle -\frac{1}{2}}\!, \, 
p^{\scriptscriptstyle -\frac{1}{2}}\!, \, z;  p \big) \big| <  107\, |z|
\end{equation*} 
\phantom{}
\begin{equation*} 
\big| f_{\scriptscriptstyle \mathrm{even}}^{\, -}\big(p^{\scriptscriptstyle -\frac{1}{2}}\!, \, p^{\scriptscriptstyle -\frac{1}{2}}\!, \, 
p^{\scriptscriptstyle -\frac{1}{2}}\!, \, z;  p \big) \big | <  564\, p^{\scriptscriptstyle -\frac{1}{2}}
\end{equation*} 
and 
\begin{equation*} 
\frac{1}{\big| f_{\scriptscriptstyle \mathrm{even}}^{\, +}\big(p^{\scriptscriptstyle -\frac{1}{2}}\!, \, 
p^{\scriptscriptstyle -\frac{1}{2}}\!, \, p^{\scriptscriptstyle -\frac{1}{2}}\!, \, z; p \big) \big|} <  25.
\end{equation*} 
\end{lem}

\begin{proof} \!By \cite{D}, we have:  
\begin{equation*}
f_{\scriptscriptstyle \mathrm{odd}}\big(p^{\scriptscriptstyle -\frac{1}{2}}\!, \, p^{\scriptscriptstyle -\frac{1}{2}}\!, \, 
p^{\scriptscriptstyle -\frac{1}{2}}\!, \, z  ; p \big) \, = \, 
\frac{z \, (1 + 7 z^{\scriptscriptstyle 2} + 7 z^{\scriptscriptstyle 4} + z^{\scriptscriptstyle 6})}
{(1 - z^{\scriptscriptstyle 2})^{\scriptscriptstyle 7}  (1 - p \, z^{\scriptscriptstyle 4})} 
\end{equation*} 
\phantom{} 
\begin{equation*}
f_{\scriptscriptstyle \mathrm{even}}^{\, -}\big(p^{\scriptscriptstyle -\frac{1}{2}}\!, \, p^{\scriptscriptstyle -\frac{1}{2}}\!, \, 
p^{\scriptscriptstyle -\frac{1}{2}}\!, \, z  ;  p \big) \, = \, 
\frac{3 + p^{\scriptscriptstyle - 1} + (10 - 17 p^{\scriptscriptstyle - 1} + 3 p^{\scriptscriptstyle - 2})\, z^{\scriptscriptstyle 2}  
+ (3  - 17 p^{\scriptscriptstyle - 1} + 10 p^{\scriptscriptstyle - 2})\, z^{\scriptscriptstyle 4}  
+ (p^{\scriptscriptstyle  -1} + 3 p^{\scriptscriptstyle - 2})\,  z^{\scriptscriptstyle 6}}{\sqrt{p}\, 
(1 - p^{\scriptscriptstyle - 1})^{\scriptscriptstyle 3}(1 - z^{\scriptscriptstyle 2})^{\scriptscriptstyle 6} 
(1 - p\,  z^{\scriptscriptstyle 4})}
\end{equation*} 
and 
\begin{equation*}
\begin{split}
& 1\slash f_{\scriptscriptstyle \mathrm{even}}^{\, +}\big(p^{\scriptscriptstyle -\frac{1}{2}}\!, \, 
p^{\scriptscriptstyle -\frac{1}{2}}\!, \, p^{\scriptscriptstyle -\frac{1}{2}}\!, \, z; p \big)  \\ 
& = \frac{(1 - p^{\scriptscriptstyle - 1})^{\scriptscriptstyle 3}(1 - z^{\scriptscriptstyle 2})^{\scriptscriptstyle 7} 
(1 - p\,  z^{\scriptscriptstyle 4})}
{1 + 3 p^{\scriptscriptstyle - 1} + (7 - 15 p^{\scriptscriptstyle - 1} + p^{\scriptscriptstyle - 2} 
- p^{\scriptscriptstyle - 3})\, z^{\scriptscriptstyle 2} + (7  - 35 p^{\scriptscriptstyle - 1} + 35 p^{\scriptscriptstyle - 2} 
- 7 p^{\scriptscriptstyle - 3})\, z^{\scriptscriptstyle 4} + (1 - p^{\scriptscriptstyle  -1} + 15 p^{\scriptscriptstyle - 2} - 7 p^{\scriptscriptstyle - 3})\,  z^{\scriptscriptstyle 6} - (3 p^{\scriptscriptstyle - 2} + p^{\scriptscriptstyle - 3})
\,  z^{\scriptscriptstyle 8}}.
\end{split}
\end{equation*} 
It follows that 
\begin{equation*} 
\big|f_{\scriptscriptstyle \mathrm{odd}}\big(p^{\scriptscriptstyle -\frac{1}{2}}\!, \, p^{\scriptscriptstyle -\frac{1}{2}}\!, \, 
p^{\scriptscriptstyle -\frac{1}{2}}\!, \, z;  p \big) \big| \, \le \, 
\frac{1 + 7 |z|^{\scriptscriptstyle 2} + 7 |z|^{\scriptscriptstyle 4} + |z|^{\scriptscriptstyle 6}}
{(1 - |z|^{\scriptscriptstyle 2})^{\scriptscriptstyle 7}  (1 - p \, |z|^{\scriptscriptstyle 4})}  \cdot |z| \, \le \, 
\frac{1 + 7 p^{\scriptscriptstyle - 1} + 7 p^{\scriptscriptstyle - 2} + p^{\scriptscriptstyle - 3}}
{(1 - p^{\scriptscriptstyle - 1})^{\scriptscriptstyle 8}}  \cdot |z|.
\end{equation*} 
The expression 
\begin{equation*}  
\frac{1 + 7 p^{\scriptscriptstyle - 1} + 7 p^{\scriptscriptstyle - 2} + p^{\scriptscriptstyle - 3}}
{(1 - p^{\scriptscriptstyle - 1})^{\scriptscriptstyle 8}} \qquad \text{(for $p \ge 3$)}
\end{equation*} 
is increasing as a function of $p^{\scriptscriptstyle - 1}\!,$ and its value when $p = 3$ is $106.312... < 107.$ Similarly    
\begin{equation*} 
\big|f_{\scriptscriptstyle \mathrm{even}}^{\, -}\big(p^{\scriptscriptstyle -\frac{1}{2}}\!, \, p^{\scriptscriptstyle -\frac{1}{2}}\!, \, 
p^{\scriptscriptstyle -\frac{1}{2}}\!, \, z; p \big) \big| 
\, \le \, \frac{3 + 11 p^{\scriptscriptstyle - 1} + 20\, p^{\scriptscriptstyle - 2} + 20\, p^{\scriptscriptstyle - 3}  
+ 11 p^{\scriptscriptstyle - 4} + 3 p^{\scriptscriptstyle - 5}}
{(1 - p^{\scriptscriptstyle - 1})^{\scriptscriptstyle 10}} \cdot p^{\scriptscriptstyle - \frac{1}{2}} <  
564 p^{\scriptscriptstyle - \frac{1}{2}}
\end{equation*} 
as we had asserted. 

The numerator of 
$
1\slash \big|f_{\scriptscriptstyle \mathrm{even}}^{\, +}\big(p^{\scriptscriptstyle -\frac{1}{2}}\!, \, 
p^{\scriptscriptstyle -\frac{1}{2}}\!, \, p^{\scriptscriptstyle -\frac{1}{2}}\!, \, z; p \big) \big|
$ 
is  
\begin{equation*}
(1 - p^{\scriptscriptstyle - 1})^{\scriptscriptstyle 3}\, |1 - z^{\scriptscriptstyle 2}|^{\scriptscriptstyle 7} \, 
|1 - p\,  z^{\scriptscriptstyle 4}| < (1 + p^{\scriptscriptstyle - 1})^{\scriptscriptstyle 8} \le 
(\tfrac{4}{3})^{\scriptscriptstyle 8}. 
\end{equation*} 
To obtain a lower bound for the denominator, we first assume that $p \ge 11.$ In this case we have 
\begin{equation*} 
\begin{split} 
& \big|1 + 3 p^{\scriptscriptstyle - 1} + (7 - 15 p^{\scriptscriptstyle - 1} + p^{\scriptscriptstyle - 2} 
- p^{\scriptscriptstyle - 3})\, z^{\scriptscriptstyle 2} + (7  - 35 p^{\scriptscriptstyle - 1} + 35 p^{\scriptscriptstyle - 2} 
- 7 p^{\scriptscriptstyle - 3})\, z^{\scriptscriptstyle 4} + (1 - p^{\scriptscriptstyle  -1} + 15 p^{\scriptscriptstyle - 2} - 7 p^{\scriptscriptstyle - 3})\,  z^{\scriptscriptstyle 6} - (3 p^{\scriptscriptstyle - 2} + p^{\scriptscriptstyle - 3})
\,  z^{\scriptscriptstyle 8} \big| \\ 
& \ge 1 + 3 p^{\scriptscriptstyle - 1} - \big|(7 - 15 p^{\scriptscriptstyle - 1} + p^{\scriptscriptstyle - 2} 
- p^{\scriptscriptstyle - 3})\, z^{\scriptscriptstyle 2} + (7  - 35 p^{\scriptscriptstyle - 1} + 35 p^{\scriptscriptstyle - 2} 
- 7 p^{\scriptscriptstyle - 3})\, z^{\scriptscriptstyle 4} + (1 - p^{\scriptscriptstyle  -1} + 15 p^{\scriptscriptstyle - 2} - 7 p^{\scriptscriptstyle - 3})\,  z^{\scriptscriptstyle 6} - (3 p^{\scriptscriptstyle - 2} + p^{\scriptscriptstyle - 3})
\,  z^{\scriptscriptstyle 8} \big| \\
& \ge 1 - 4 p^{\scriptscriptstyle - 1} - 22\, p^{\scriptscriptstyle - 2} - 37 p^{\scriptscriptstyle - 3} - 
37 p^{\scriptscriptstyle - 4} - 22 p^{\scriptscriptstyle - 5} - 10\, p^{\scriptscriptstyle - 6} - p^{\scriptscriptstyle - 7} 
> \tfrac{2}{5}.
\end{split}
\end{equation*} 
When $p = 3$ we have   
\begin{equation*} 
\left|2 - \tfrac{2 \, z^{\scriptscriptstyle 2}}{27} 
(5 z^{\scriptscriptstyle 6} - 28 z^{\scriptscriptstyle 4} + 14 z^{\scriptscriptstyle 2} - 28)\right| \ge 
2 - \tfrac{2 \, |z|^{\scriptscriptstyle 2}}{27} 
\big|5 z^{\scriptscriptstyle 6} - 28 z^{\scriptscriptstyle 4} + 14 z^{\scriptscriptstyle 2} - 28 \big| 
\ge 2 - \tfrac{2}{81}\left(28 + \tfrac{14}{3} + \tfrac{28}{9} + \tfrac{5}{27}\right) > 
\tfrac{2}{5}.
\end{equation*} 
When $p = 5$ we have   
\begin{equation*} 
\left|\tfrac{8}{5} - \tfrac{8 \, z^{\scriptscriptstyle 2}}{125} 
(2 z^{\scriptscriptstyle 6} - 21 z^{\scriptscriptstyle 4} - 21 z^{\scriptscriptstyle 2} - 63) \right| \ge 
\tfrac{8}{5} - \tfrac{8 \, |z|^{\scriptscriptstyle 2}}{125} 
\big| 2 z^{\scriptscriptstyle 6} - 21 z^{\scriptscriptstyle 4} - 21 z^{\scriptscriptstyle 2} - 63 \big| 
\ge \tfrac{8}{5} - \tfrac{8}{625}\left(63 + \tfrac{21}{5} + \tfrac{21}{25} + \tfrac{2}{125}\right) > 
\tfrac{2}{5}.
\end{equation*} 
Similarly, when $p = 7$ we have 
\begin{equation*} 
\left|\tfrac{10}{7} - \tfrac{2 \, z^{\scriptscriptstyle 2}}{343} 
(11 z^{\scriptscriptstyle 6} - 196 z^{\scriptscriptstyle 4} - 462 z^{\scriptscriptstyle 2} - 836) \right| > 
\tfrac{2}{5}.
\end{equation*} 
The last assertion follows from these inequalities.  
\end{proof}

For ease of notation, we define 
\begin{equation*}
\tilde{\mathscr{Z}}^{(c)}
(s; \chi_{a_{\scriptscriptstyle 2}c_{\scriptscriptstyle 2}}, 
\chi_{a_{\scriptscriptstyle 1}c_{\scriptscriptstyle 1}}) \, = \, 
\frac{(s - 1)^{\scriptscriptstyle 7} \big(s - \frac{3}{4}\big)}{(s + 1)^{\scriptscriptstyle 8}} \cdot
Z^{(c)}\big(\tfrac{1}{2}, \tfrac{1}{2}, \tfrac{1}{2}, s; 
\chi_{a_{\scriptscriptstyle 2}c_{\scriptscriptstyle 2}}, 
\chi_{a_{\scriptscriptstyle 1}c_{\scriptscriptstyle 1}}\big).
\end{equation*}

\begin{prop}\label{key-proposition} --- Let $c_{\scriptscriptstyle 1}, c_{\scriptscriptstyle 2}$ and 
$c_{\scriptscriptstyle 3}$ be odd positive integers such that 
$
c = c_{\scriptscriptstyle 1}c_{\scriptscriptstyle 2}c_{\scriptscriptstyle 3}
$ 
is square-free, \!and let $\omega(c_{\scriptscriptstyle i})$ denote the number of prime factors of 
$c_{\scriptscriptstyle i},$ for $1\le i \le 3.$ Then, for every $\delta > 0$ and 
$a_{\scriptscriptstyle 1}, \, a_{\scriptscriptstyle 2} \in \{\pm 1, \pm 2\},$ we have the estimate 
\begin{equation}  \label{eq: fundamental-estimate-optimized} 
\tilde{\mathscr{Z}}^{(c)}
(s; \chi_{a_{\scriptscriptstyle 2}c_{\scriptscriptstyle 2}}, 
\chi_{a_{\scriptscriptstyle 1}c_{\scriptscriptstyle 1}})  
\, \ll_{\scriptscriptstyle \delta} \, 
(1+ |s|)^{\scriptscriptstyle 5(1 - \Re(s)) \, + \, \delta} 
A_{0}^{\omega(c_{\scriptscriptstyle 1}c_{\scriptscriptstyle 2})} A_{1}^{\omega(c_{\scriptscriptstyle 3})} S(c, \delta)
c_{\scriptscriptstyle 1}^{\scriptscriptstyle 3 (1 - \Re(s))} 
\, c_{\scriptscriptstyle 2}^{\scriptscriptstyle \frac{5}{2}(1 - \Re(s))} 
\, c_{\scriptscriptstyle 3}^{\scriptscriptstyle \mathrm{max}\left\{3 \, - \, 4 \Re(s), \, 2 \, - \frac{5 \Re(s)}{2}\right\}}
c^{\scriptscriptstyle \delta}
\end{equation} 
with $A_{1}\! = 25 + 16775 \, A_{0}$ and  
\begin{equation*} 
S(c, \delta) \; =  
\sum_{a \, = \, \pm 1, \, \pm 2}\;\, \sum_{b \, \mid \, c} 
\; \sum_{(d_{\scriptscriptstyle 0}, \, 2) \, = \, 1} 
\, \big|L^{\scriptscriptstyle (2)}\big(\tfrac{1}{2},
\chi_{a b d_{\scriptscriptstyle 0}} \big)\big|^{3} d_{\scriptscriptstyle 0}^{\scriptscriptstyle - 1 - (\delta \slash 30)}
\end{equation*} 
for all $s$ with $\frac{1}{2} \le \Re(s) \le \frac{4}{5}.$
\end{prop}

\begin{proof} \!\!As in the proof of \cite[Proposition~6.3]{D}, we proceed by induction on 
$\omega(c_{\scriptscriptstyle 3}).$ If $c_{\scriptscriptstyle 3} = 1$ then, for every 
$\delta > 0,$ $c_{\scriptscriptstyle 1}, c_{\scriptscriptstyle 2}$ odd positive integers such that 
$
c_{\scriptscriptstyle 1}c_{\scriptscriptstyle 2}
$ 
is square-free, and $s$ with $\frac{1}{2} \le \Re(s) \le \frac{4}{5},$ we have from the bound 
\eqref{eq: basic-initial-estimate} that 
\begin{equation*}
\big|\tilde{\mathscr{Z}}^{(c_{\scriptscriptstyle 1}c_{\scriptscriptstyle 2})}
(s; \chi_{a_{\scriptscriptstyle 2}c_{\scriptscriptstyle 2}}, 
\chi_{a_{\scriptscriptstyle 1}c_{\scriptscriptstyle 1}})\big|  
\, \le  \, B(\delta) \, (1+ |s|)^{\scriptscriptstyle 5(1 - \Re(s)) \, + \, \delta} 
A_{0}^{\omega(c_{\scriptscriptstyle 1}c_{\scriptscriptstyle 2})} 
S(c_{\scriptscriptstyle 1}c_{\scriptscriptstyle 2}, \delta)
c_{\scriptscriptstyle 1}^{\scriptscriptstyle 3 (1 - \Re(s))} 
\, c_{\scriptscriptstyle 2}^{\scriptscriptstyle \frac{5}{2}(1 - \Re(s))} 
\, (c_{\scriptscriptstyle 1}c_{\scriptscriptstyle 2})^{\scriptscriptstyle \delta}
\end{equation*} 
for some $B(\delta) > 0.$ 

To carry out the inductive step, suppose $p$ is an odd prime with $p \nmid c$. Then (as in \cite{D}) we can write 
\begin{equation*} 
\begin{split}
\tilde{\mathscr{Z}}^{(c)}(s; 
\chi_{a_{\scriptscriptstyle 2}c_{\scriptscriptstyle 2}}, \chi_{a_{\scriptscriptstyle 1}c_{\scriptscriptstyle 1}}) 
& = \chi_{a_{\scriptscriptstyle 2}c_{\scriptscriptstyle 2}}(p)\,
\tilde{\mathscr{Z}}^{(c p)}
(s; \chi_{a_{\scriptscriptstyle 2}c_{\scriptscriptstyle 2}}, 
\chi_{a_{\scriptscriptstyle 1}c_{\scriptscriptstyle 1}p})\, f_{\scriptscriptstyle \mathrm{odd}}
\big(p^{\scriptscriptstyle - \frac{1}{2}}\!, \, p^{\scriptscriptstyle - \frac{1}{2}}\!, \, p^{\scriptscriptstyle - \frac{1}{2}}\!, 
\, p^{- s}; p \big)\\
&+ \, \chi_{a_{\scriptscriptstyle 1}c_{\scriptscriptstyle 1}}\!(p)\, 
\tilde{\mathscr{Z}}^{(c p)}(s; \chi_{a_{\scriptscriptstyle 2}c_{\scriptscriptstyle 2} p^{\scriptscriptstyle *}},
\chi_{a_{\scriptscriptstyle 1}c_{\scriptscriptstyle 1}})\, 
f_{\scriptscriptstyle \mathrm{even}}^{-}
\big(p^{\scriptscriptstyle - \frac{1}{2}}\!,  \, p^{\scriptscriptstyle - \frac{1}{2}}\!, \, p^{\scriptscriptstyle - \frac{1}{2}}\!, 
\, p^{- s}; p \big)\\ 
& + \, \tilde{\mathscr{Z}}^{(c p)}
(s; \chi_{a_{\scriptscriptstyle 2}c_{\scriptscriptstyle 2}}, 
\chi_{a_{\scriptscriptstyle 1}c_{\scriptscriptstyle 1}})\, 
f_{\scriptscriptstyle \mathrm{even}}^{+}
\big(p^{\scriptscriptstyle - \frac{1}{2}}\!,  \, p^{\scriptscriptstyle - \frac{1}{2}}\!, \, p^{\scriptscriptstyle - \frac{1}{2}}\!, 
\, p^{- s}; p \big)
\end{split}
\end{equation*} 
where 
$
p^{\scriptscriptstyle *} : = (- 1)^{\scriptscriptstyle (p \, - \, 1)\slash 2} p.
$ 
Applying the inequalities in Lemma \ref{Estimate inverse even part zeven}, it follows that, for $\Re(s) \ge \frac{1}{2},$
\begin{equation*}
\begin{split} 
\big|\tilde{\mathscr{Z}}^{(c p)} (s; \chi_{a_{\scriptscriptstyle 2}c_{\scriptscriptstyle 2}}, 
\chi_{a_{\scriptscriptstyle 1}c_{\scriptscriptstyle 1}})\big|  
& <  25\, \big|\tilde{\mathscr{Z}}^{(c)}(s; 
\chi_{a_{\scriptscriptstyle 2}c_{\scriptscriptstyle 2}}, 
\chi_{a_{\scriptscriptstyle 1}c_{\scriptscriptstyle 1}}) \big| + 107\cdot 25\, \big|\tilde{\mathscr{Z}}^{(c p)}(s; \chi_{a_{\scriptscriptstyle 2}c_{\scriptscriptstyle 2}}, 
\chi_{a_{\scriptscriptstyle 1}c_{\scriptscriptstyle 1}p}) \big|\, p^{\scriptscriptstyle - \Re(s)} \\
& + 564\cdot 25\, \big|\tilde{\mathscr{Z}}^{(c p)}(s; \chi_{a_{\scriptscriptstyle 2}c_{\scriptscriptstyle 2} p^{\scriptscriptstyle *}},
\chi_{a_{\scriptscriptstyle 1}c_{\scriptscriptstyle 1}}) \big|\, 
p^{\scriptscriptstyle -\frac{1}{2}}.
\end{split}
\end{equation*} 
Let $K(c_{\scriptscriptstyle 1}, c_{\scriptscriptstyle 2}, c_{\scriptscriptstyle 3}, s, \delta)$ denote the right-hand side of 
\eqref{eq: fundamental-estimate-optimized}, i.e., 
\begin{equation*}
K(c_{\scriptscriptstyle 1}, c_{\scriptscriptstyle 2}, c_{\scriptscriptstyle 3}, s, \delta) 
= B(\delta) \, (1+ |s|)^{\scriptscriptstyle 5(1 - \Re(s)) \, + \, \delta} 
A_{0}^{\omega(c_{\scriptscriptstyle 1}c_{\scriptscriptstyle 2})} A_{1}^{\omega(c_{\scriptscriptstyle 3})} S(c, \delta)
c_{\scriptscriptstyle 1}^{\scriptscriptstyle 3 (1 - \Re(s))} 
\, c_{\scriptscriptstyle 2}^{\scriptscriptstyle \frac{5}{2}(1 - \Re(s))} 
\, c_{\scriptscriptstyle 3}^{\scriptscriptstyle \mathrm{max}\left\{3 \, - \, 4 \Re(s), \, 2 \, - \frac{5 \Re(s)}{2}\right\}}
c^{\scriptscriptstyle \delta}
\end{equation*} 
and note that $S(c, \delta) \le S(c p, \delta).$ Taking $s$ such that $\frac{1}{2} \le \Re(s) \le \frac{4}{5},$ we have by the induction hypothesis that 
\begin{equation*} 
\begin{split}
\big|\tilde{\mathscr{Z}}^{(c p)} (s; \chi_{a_{\scriptscriptstyle 2}c_{\scriptscriptstyle 2}}, 
\chi_{a_{\scriptscriptstyle 1}c_{\scriptscriptstyle 1}})\big| \, & < \, 
K(c_{\scriptscriptstyle 1}, c_{\scriptscriptstyle 2}, c_{\scriptscriptstyle 3}, s, \delta)  
\frac{S(c p, \delta)}{S(c, \delta)}\cdot \Big(25 + 2675 \, A_{0} \, p^{\scriptscriptstyle 3 \, - \, 4 \Re(s) + \delta} + 
14100\, A_{0} \, p^{\scriptscriptstyle 2 \, - \frac{5\Re(s)}{2} + \, \delta}\Big)\\
& < \, K(c_{\scriptscriptstyle 1}, c_{\scriptscriptstyle 2}, c_{\scriptscriptstyle 3}p, s, \delta) 
\end{split}
\end{equation*} 
and the proposition follows.
\end{proof} 

\vskip4pt 
Using the last proposition, we can now estimate the function 
\begin{equation*}
\tilde{\mathscr{Z}}(s, \chi_{a_{\scriptscriptstyle 2}}, \chi_{a_{\scriptscriptstyle 1}}\!; \, h) \, := \, 
\frac{(s - 1)^{\scriptscriptstyle 7} \big(s - \frac{3}{4}\big)}{(s + 1)^{\scriptscriptstyle 8}} \cdot 
Z\big(\tfrac{1}{2}, \tfrac{1}{2}, \tfrac{1}{2}, s, \chi_{a_{\scriptscriptstyle 2}}, \chi_{a_{\scriptscriptstyle 1}}\!; \, h \big).
\end{equation*} 

\begin{thm}\label{Main-h-Estimate} --- For any square-free odd positive integer $h,$ 
$a_{\scriptscriptstyle 1}, \, a_{\scriptscriptstyle 2} \in \{\pm 1, \pm 2\},$ and every $\delta > 0,$ 
we have 
\begin{equation*} 
\tilde{\mathscr{Z}}(s, \chi_{a_{\scriptscriptstyle 2}}, \chi_{a_{\scriptscriptstyle 1}}\!; \, h)\, 
\ll_{\delta}\,     (1+ |s|)^{\scriptscriptstyle 5(1 - \Re(s)) \, + \, \delta} 
\, S(h, \delta)\, h^{\scriptscriptstyle 2 \, - \frac{9 \Re(s)}{2} + \, 2 \delta}
\end{equation*} 
on the strip $\frac{2}{3} \le \Re(s) \le \frac{4}{5},$ and 
\begin{equation*} 
h^{2 s} \tilde{\mathscr{Z}}(s, \chi_{a_{\scriptscriptstyle 2}}, \chi_{a_{\scriptscriptstyle 1}}\!; \, h) \,  \ll_{\delta} \, 
(1+ |s|)^{\scriptscriptstyle 5(1 - \Re(s)) \, + \, \delta} 
\, S\big(h, \tfrac{\delta}{5}\big) \, h^{\scriptscriptstyle 2 \delta} 
\end{equation*} 
on the strip $\frac{4}{5} \le \Re(s) \le 1+ \frac{\delta}{5}.$ 
\end{thm}

\begin{proof} \!The proof is similar to that given in \cite[Theorem~6.4]{D}. By \eqref{eq: fundamental-eq-h} we have 
\begin{equation*}
\begin{split}
|\tilde{\mathscr{Z}}(s, \chi_{a_{\scriptscriptstyle 2}}, \chi_{a_{\scriptscriptstyle 1}}\!; \, h)| \, \le \,
h^{- 2 \Re(s)} \!\!\sum_{h = c_{\scriptscriptstyle 1}c_{\scriptscriptstyle 2}c_{\scriptscriptstyle 3}} \; 
& \big|\tilde{\mathscr{Z}}^{(h)}(s; \chi_{a_{\scriptscriptstyle 2}}\tilde{\chi}_{c_{\scriptscriptstyle 2}}\!, 
\chi_{a_{\scriptscriptstyle 1}c_{\scriptscriptstyle 1}})\big| 
\prod_{p \, \mid \, c_{\scriptscriptstyle 1}} \big|F\big(p^{\, - \frac{1}{2}}\!, p^{\, - \frac{1}{2}}\!, p^{\, - \frac{1}{2}}\!, p^{\, - s}; p\big) \big| p^{\, - \Re(s)} \\
&\cdot \prod_{p \, \mid \, c_{\scriptscriptstyle 2}} \big|G^{(1)}\big(p^{\, - \frac{1}{2}}\!, p^{\, - \frac{1}{2}}\!, p^{\, - \frac{1}{2}}\!, p^{\, - s}; p\big) \big|
\prod_{p \, \mid \, c_{\scriptscriptstyle 3}} \big|G^{(0)}\big(p^{\, - \frac{1}{2}}\!, p^{\, - \frac{1}{2}}\!, p^{\, - \frac{1}{2}}\!, p^{\, - s}; p\big) \big|.
\end{split}
\end{equation*} 
For each $a_{\scriptscriptstyle 2} \in \{\pm 1, \pm 2\},$ we can write 
$
\chi_{a_{\scriptscriptstyle 2}}\tilde{\chi}_{c_{\scriptscriptstyle 2}} = 
\chi_{a_{\scriptscriptstyle 2}' c_{\scriptscriptstyle 2}} 
$ 
for some $a_{\scriptscriptstyle 2}' \in \{\pm 1, \pm 2\}.$ By \cite[Lemma~6.1]{D}, \!we have the estimates 
\begin{equation*}
\begin{split} 
&F\big(p^{\scriptscriptstyle - \frac{1}{2}}\!, \, p^{\scriptscriptstyle - \frac{1}{2}}\!, \, 
p^{\scriptscriptstyle - \frac{1}{2}}\!, \, p^{\, - s}; \, p \big) = 14 \, + \,  p^{\scriptscriptstyle 1 - 2 s} \, + \, O(p^{\scriptscriptstyle  - 2 \Re(s)})\\
&G^{\scriptscriptstyle (0)}\big(p^{\scriptscriptstyle - \frac{1}{2}}\!, \, p^{\scriptscriptstyle - \frac{1}{2}}\!, \, p^{\scriptscriptstyle - \frac{1}{2}}\!, 
\, p^{\, - s}; \, p \big) = 14 \, + \,  p^{\scriptscriptstyle 1 - 2 s}\, + \, O(p^{\scriptscriptstyle -1})\\
&G^{\scriptscriptstyle (1)}\big(p^{\scriptscriptstyle - \frac{1}{2}}\!, \, p^{\scriptscriptstyle - \frac{1}{2}}\!, \, p^{\scriptscriptstyle - \frac{1}{2}}\!, 
\, p^{\, - s}; \, p \big) = O\big(p^{\scriptscriptstyle - \frac{1}{2}}\big)
\end{split}
\end{equation*} 
the implied constants in the $O$-symbols being independent of $s, p.$ Applying Proposition \ref{key-proposition}, 
we see that, for every $s$ in the strip $\frac{2}{3} \le \Re(s) \le \frac{4}{5},$ and $\delta > 0,$ we have 
\begin{equation*}
\begin{split}
\tilde{\mathscr{Z}}(s, \chi_{a_{\scriptscriptstyle 2}}, \chi_{a_{\scriptscriptstyle 1}}\!; \, h) \,  &\ll_{\delta} \, 
(1+ |s|)^{\scriptscriptstyle 5(1 - \Re(s)) \, + \, \delta} \, B^{\omega(h)} S(h, \delta) 
\, h^{\scriptscriptstyle 2 \, - \frac{9 \Re(s)}{2} + \, \delta} 
\sum_{h = c_{\scriptscriptstyle 1}c_{\scriptscriptstyle 2}c_{\scriptscriptstyle 3}} 1 \\
& \ll_{\delta} \, (1+ |s|)^{\scriptscriptstyle 5(1 - \Re(s)) \, + \, \delta} \, (3 B)^{\omega(h)} S(h, \delta) 
\, h^{\scriptscriptstyle 2 \, - \frac{9 \Re(s)}{2} + \, \delta}
\end{split}
\end{equation*} 
for some positive constant $B.$ In particular, if $\Re(s) = \frac{4}{5}$ we have  
\begin{equation*} 
h^{2 s} \tilde{\mathscr{Z}}(s, \chi_{a_{\scriptscriptstyle 2}}, \chi_{a_{\scriptscriptstyle 1}}\!; \, h) \,  \ll_{\delta} \, 
(1+ |s|)^{\scriptscriptstyle 1 \, + \, \delta} \, (3 B)^{\omega(h)} S(h, \delta) 
\, h^{\scriptscriptstyle \delta}.
\end{equation*} 
On the other hand, if $\Re(s) = 1 + \frac{\delta}{5}$ we trivially have (by \eqref{eq: basic-initial-estimate}) that 
\begin{equation*} 
h^{2 s} \tilde{\mathscr{Z}}(s, \chi_{a_{\scriptscriptstyle 2}}, \chi_{a_{\scriptscriptstyle 1}}\!; \, h) 
\,  \ll_{\delta} \, B_{1}^{\omega(h)} S\big(h, \tfrac{\delta}{5}\big) \, h^{\scriptscriptstyle \delta} 
\end{equation*} 
for some computable positive constant $B_{1},$ and the theorem follows by applying the Phragmen-Lindel\"of 
principle, and the well-known estimate 
\begin{equation*}
\omega(h) \ll \frac{\log \, h}{\log \, \log \, h}.
\end{equation*} 
\end{proof}

\section{Proofs of Main Theorems} \label{Proofs} 
{\bf Proof of Theorem \ref{Main Theorem A}.} \!The function 
$ 
\tilde{\mathscr{Z}}_{\scriptscriptstyle 0}(s, \chi_{a_{\scriptscriptstyle 2}}, \chi_{a_{\scriptscriptstyle 1}}) : = 
(s + 1)^{\scriptscriptstyle - 8} (s - 1)^{\scriptscriptstyle 7} \big(s - \frac{3}{4}\big)
Z_{\scriptscriptstyle 0}\big(\tfrac{1}{2}, \tfrac{1}{2}, \tfrac{1}{2}, s, 
\, \chi_{a_{\scriptscriptstyle 2}}, \chi_{a_{\scriptscriptstyle 1}}\big) 
$ 
is holomorphic in the half-plane $\Re(s) > 1,$ and in this region we have 
\begin{equation} \label{eq: sqfree-non-sqfree1/2}
\tilde{\mathscr{Z}}_{\scriptscriptstyle 0}(s, \chi_{a_{\scriptscriptstyle 2}}, \chi_{a_{\scriptscriptstyle 1}})\; = 
\sum_{h - \mathrm{odd}}\,  \mu(h) 
\tilde{\mathscr{Z}}(s, \chi_{a_{\scriptscriptstyle 2}}, \chi_{a_{\scriptscriptstyle 1}}\!; \, h).
\end{equation} 
Take $s$ such that 
$
\frac{2}{3} + \delta_{\scriptscriptstyle 0} < \Re(s) < \frac{4}{5},
$ 
for a small $\delta_{\scriptscriptstyle 0} > 0,$ and let $0 < \delta < \frac{9 \delta_{\scriptscriptstyle 0}}{4}.$ By Theorem \ref{Main-h-Estimate} and the definition of 
$S(h, \delta),$ we have that 
\begin{equation*} 
\begin{split}
&\sum_{h - \mathrm{odd \; \& \; sq. \; free}}\, 
|\tilde{\mathscr{Z}}(s, \chi_{a_{\scriptscriptstyle 2}}, \chi_{a_{\scriptscriptstyle 1}}\!; \, h)| \\
&\ll_{\scriptscriptstyle \delta} \, (1+ |s|)^{\scriptscriptstyle 5(1 - \Re(s)) \, + \, \delta} \sum_{a \, = \, \pm 1, \, \pm 2}\;\, 
\sum_{h - \mathrm{odd \; \& \; sq. \; free}} 
h^{\scriptscriptstyle 2 \, - \frac{9 \Re(s)}{2} + \, \delta} \, \sum_{b \, \mid \, h} \; \sum_{(d_{\scriptscriptstyle 0}, \, 2) \, = \, 1} 
\, \big|L^{\scriptscriptstyle (2)}\big(\tfrac{1}{2},
\chi_{a b d_{\scriptscriptstyle 0}} \big)\big|^{3} d_{\scriptscriptstyle 0}^{\scriptscriptstyle - 1 - (\delta \slash 60)}\\
&\ll_{\scriptscriptstyle \delta} \, (1 + |s|)^{\scriptscriptstyle 5(1 - \Re(s)) \, + \, \delta} \sum_{a \, = \, \pm 1, \, \pm 2} \;\; 
\sum_{b, \, d_{\scriptscriptstyle 0} - \mathrm{odd \; \& \; sq. \; free}}  \, \big|L\big(\tfrac{1}{2}, \chi_{a b d_{\scriptscriptstyle 0}} \big)\big|^{3} 
b^{\scriptscriptstyle 2 \, - \frac{9 \Re(s)}{2} + \, \delta}d_{\scriptscriptstyle 0}^{\scriptscriptstyle - 1 - (\delta \slash 60)} 
\, \sum_{m} \, m^{\scriptscriptstyle 2 \, - \frac{9 \Re(s)}{2} + \, \delta}.
\end{split}
\end{equation*} 
The series $\sum \, m^{\scriptscriptstyle 2 \, - \frac{9 \Re(s)}{2} + \, \delta}$ is convergent, and for $a \in \{\pm 1, \, \pm 2\},$
\begin{equation*} 
\sum_{b, \, d_{\scriptscriptstyle 0} - \mathrm{odd \; \& \; sq. \; free}}\, \big|L\big(\tfrac{1}{2}, \chi_{a b d_{\scriptscriptstyle 0}} \big)\big|^{3} 
b^{\scriptscriptstyle 2 \, - \frac{9 \Re(s)}{2} + \, \delta}d_{\scriptscriptstyle 0}^{\scriptscriptstyle - 1 - (\delta \slash 60)}  \, < \, 
\sum_{(n, 2) = 1} d(n)\, \big|L\big(\tfrac{1}{2}, \chi_{a n} \big)\big|^{3} n^{\scriptscriptstyle - 1 - (\delta \slash 60)}. 
\end{equation*} 
The last series is easily seen to be convergent by Heath-Brown's estimate \cite{H-B}. Thus the right-hand side of 
\eqref{eq: sqfree-non-sqfree1/2} conver-\\ges absolutely and uniformly on every compact subset of the strip 
$
\frac{2}{3} <  \Re(s) < \frac{4}{5}.
$

In a completely analogous fashion, one shows that the right-hand side of \eqref{eq: sqfree-non-sqfree1/2} is 
convergent absolutely and uniformly on every compact subset of a strip 
$
\frac{4}{5} - \delta_{\scriptscriptstyle 0} < \Re(s) < 1 + \delta_{\scriptscriptstyle 0},
$ 
for small positive $\delta_{\scriptscriptstyle 0},$ which, by Weierstrass Theorem, completes the analytic continuation 
of the function 
$
\tilde{\mathscr{Z}}_{\scriptscriptstyle 0}(s, \chi_{a_{\scriptscriptstyle 2}}, \chi_{a_{\scriptscriptstyle 1}})
$ 
to the half-plane $\Re(s) > 2 \slash 3.$

We set $a_{\scriptscriptstyle 1} = 2$ and $a_{\scriptscriptstyle 2} = 1,$ hence 
$
\chi_{a_{\scriptscriptstyle 1}}\! \chi_{\scriptscriptstyle d_{\scriptscriptstyle 0}}(n) = 
\left(\!\frac{8 d_{\scriptscriptstyle 0}}{n} \!\right)
$ 
for $n$ odd. It only remains to compute the residue of $Z_{\scriptscriptstyle 0}(s)$ at $s = \frac{3}{4}.$ 
By \eqref{mobius}, \eqref{eq: fundamental-eq-h}, and \eqref{mds-residue}, this residue has the form:
\begin{align*}
\underset{s = \frac{3}{4}}{\Res} \; Z_{\scriptscriptstyle 0}(s) \, = \, &\tfrac{9}{256\pi} 2^{\frac{1}{4}} (-181+128\sqrt{2}) 
\Gamma\left(\tfrac{1}{4}\right)^{\! 4} \zeta\left(\tfrac{1}{2}\right)^{\! 7} \\
&\cdot \sum_{\substack{h - \text{odd} \\ h = c_{\scriptscriptstyle 1}c_{\scriptscriptstyle 2}c_{\scriptscriptstyle 3}}} \, \mu(h) h^{- \frac{3}{2}} 
c_{\scriptscriptstyle 1}^{- \frac{1}{4}}
\prod_{p\mid c_{\scriptscriptstyle 1}} \big(1 - p^{- \frac{1}{2}} \big)^{\! 8} \big(1 + p^{- \frac{1}{2}}\big)^{\! 2} \big(1 + 6p^{- \frac{1}{2}} + p^{-1}\big) 
F\big(p^{\, - \frac{1}{2}}\!, p^{\, - \frac{1}{2}}\!, p^{\, - \frac{1}{2}}\!, p^{\, - \frac{3}{4}}; p\big) p^{\, - \frac{3}{4}} \\
&\hskip67pt \cdot c_{\scriptscriptstyle 2}^{- \frac{1}{2}} 
\prod_{p\mid c_{\scriptscriptstyle 2}} \big(1 - p^{- \frac{1}{2}}\big)^{\! 8} \big(1 + p^{- \frac{1}{2}}\big)\big(3 + 7p^{- \frac{1}{2}} + 3p^{-1}\big) 
G^{(1)}\big(p^{\, - \frac{1}{2}}\!, p^{\, - \frac{1}{2}}\!, p^{\, - \frac{1}{2}}\!, p^{\, - \frac{3}{4}}; p\big)\\
&\hskip50pt \cdot \prod_{p\mid c_{\scriptscriptstyle 3}} \big(1 - p^{- \frac{1}{2}}\big)^{\! 8} \big(1 + p^{- \frac{1}{2}}\big) 
\big(1 + 7p^{- \frac{1}{2}} + 13p^{-1} + 7p^{- \frac{3}{2}} + p^{-2}\big) 
G^{(0)}\big(p^{\, - \frac{1}{2}}\!, p^{\, - \frac{1}{2}}\!, p^{\, - \frac{1}{2}}\!, p^{\, - \frac{3}{4}}; p\big).
\end{align*} 
The sum over $h, c_{\scriptscriptstyle 1}, c_{\scriptscriptstyle 2}, c_{\scriptscriptstyle 3}$ is Eulerian. 
The factor at an odd prime $p$ can be computed as the sum of four explicit rational functions, corresponding to the cases 
$p \nmid h,$ $p\mid c_{\scriptscriptstyle 1},$ $p\mid c_{\scriptscriptstyle 2},$ and $p\mid c_{\scriptscriptstyle 3}.$ 
After some cancellation, this factor is
\begin{equation*}
\big(1 - p^{- \frac{1}{2}}\big)^{\! 5} \big(1 + p^{- \frac{1}{2}}\big) 
\big(1 + 4p^{- \frac{1}{2}} + 11p^{-1} + 10p^{- \frac{3}{2}} - 11p^{-2} + 11p^{-3} - 4p^{- \frac{7}{2}} - p^{- 4} \big)
\end{equation*}
which we denote as $P(p^{- \frac{1}{2}}).$ This completes the proof.

\vskip10pt
{\bf Proof of Theorem \ref{Main Theorem B}.} The argument is standard, and is included for the sake of completeness. The Mellin transform of $W,$ 
\begin{equation*} 
\widehat{W}(s) \; = \int_{0}^{\infty} W(u) u^{s}\, \frac{du}{u}
\end{equation*} 
is entire, and by using the bounds \eqref{eq: weight-function} and 
integration by parts we have the estimate
\begin{equation}  \label{eq: Mellin-weight-estimate}
|\widehat{W}(s)| \, < \, \frac{1}{3 + \Re(s)}\cdot \frac{1}{|s| \, |s + 1| \, |s + 2|} \qquad \text{(when $\Re(s) \, > \, - 3$).}
\end{equation} 
Applying the Mellin inversion formula, we can express
\begin{equation*} 
\sideset{}{^*}\sum_{(d, \, 2) \, = \, 1} 
L\big(\tfrac{1}{2}, \chi_{\scriptscriptstyle 2 d} \big)^{\! \scriptscriptstyle 3}\, 
W\big(\tfrac{d}{x}\big) \, = \, \frac{1}{2 \, \pi \, i}\int\limits_{(2)} \widehat{W}(s) Z_{\scriptscriptstyle 0}(s)\, x^{s}\, ds.
\end{equation*} 
Since 
$
Z_{\scriptscriptstyle 0}(s)  \ll_{\scriptscriptstyle \delta} 
\mathrm{max}\{1, (\! 1 + |s|)^{\scriptscriptstyle 5(1 - \Re(s)) \, + \,  \delta}\},
$ 
it follows from the upper bound estimate \eqref{eq: Mellin-weight-estimate} that we can shift the line of integration 
to $\Re(s) = 2\slash 3 + \delta.$ Thus
\begin{equation*} 
\sideset{}{^*}\sum_{(d, \, 2) \, = \, 1} 
L\big(\tfrac{1}{2}, \chi_{\scriptscriptstyle 2 d} \big)^{\! \scriptscriptstyle 3}\, 
W\big(\tfrac{d}{x}\big) \, = \, x\, \underset{s = 1}{\mathrm{Res}}
\big(\widehat{W}(s) Z_{\scriptscriptstyle 0}(s)\, x^{s - 1}\big) \, + \, 
\underset{s = \frac{3}{4}}{\mathrm{Res}} \; Z_{\scriptscriptstyle 0}(s) \cdot
\widehat{W}\big(\tfrac{3}{4}\big) \, x^{\scriptscriptstyle \frac{3}{4}}
\, + \; O_{\scriptscriptstyle \delta}\left(x^{\scriptscriptstyle \frac{2}{3} + \delta}\right)
\end{equation*} 
and the theorem follows.

%


\end{document}